\documentclass[12pt,reqno]{amsart}

\usepackage{color}
\usepackage{amsmath, amsthm, amssymb}
\usepackage{amsfonts}
\usepackage[ansinew]{inputenc}
\usepackage{float}
\usepackage{graphicx}
\usepackage[english]{babel}
\usepackage{hyperref}
\usepackage{bm}
\usepackage{tikz}
\usepgflibrary{shapes.geometric}
\usetikzlibrary{patterns,decorations.pathmorphing}
\usepackage{cite}
\usepackage{graphicx}
\usepackage{amscd}
\usepackage{color}
\usepackage{bm}
\usepackage{enumerate}

\usepackage{verbatim}
\usepackage{hyperref}
\usepackage{amstext}
\usepackage{latexsym}
\theoremstyle{plain}
\newtheorem{Theorem}{Theorem}[section]

\newtheorem{Proposition}[Theorem]{Proposition}
\newtheorem{Lemma}[Theorem]{Lemma}
\newtheorem{Remark}[Theorem]{Remark}

\numberwithin{Theorem}{section}
\numberwithin{equation}{section}

\def\proof{\noindent{{\bf Proof. }}}
\def\square{\vbox{
\hrule height .4pt \hbox{\vrule width .4pt height 7pt \kern 7pt
\vrule width .4pt} \hrule height .4pt }}
\def\QED{\hfill {$\square$}\goodbreak \medskip}

\newcommand{\average}{{\mathchoice {\kern1ex\vcenter{\hrule height.4pt
width 6pt depth0pt} \kern-9.7pt} {\kern1ex\vcenter{\hrule
height.4pt width 4.3pt depth0pt} \kern-7pt} {} {} }}

\def\R{\mathbb{R}}

\renewcommand{\a }{\alpha }

\newcommand{\e }{\varepsilon }
\newcommand{\g }{\gamma}

\renewcommand{\l }{\lambda }

\newcommand{\vp }{\varphi }

\newcommand{\s }{\sigma }

\renewcommand{\O }{\Omega }

\newcommand{\ov}{\overline}

\newcommand{\be}{\begin{equation}}
\newcommand{\ee}{\end{equation}}

\newcommand{\ti}{\widetilde}

\newcommand{\calX}{{\mathcal X}}
\newcommand{\calY}{{\mathcal Y}}

\newcommand{\mbL }{\mathbb{L}}

\DeclareMathOperator{\inv}{inv}

\newcommand{\N}{\mathbb{N}}

\newcommand{\Z}{\mathbb{Z}}

\newcommand{\cI}{{\mathcal I}}

\newcommand{\cL}{{\mathcal L}}
\newcommand{\cM}{{\mathcal M}}

\newcommand{\cO}{{\mathcal O}}

\newcommand{\cU}{{\mathcal U}}

\newcommand{\cZ}{{\mathcal Z}}

\newcommand{\eps}{\varepsilon}

\DeclareMathOperator{\id}{id}

\renewcommand{\epsilon}{\varepsilon}

\begin{document}

\title[Serrin's overdetermined problem on the sphere]
{Serrin's overdetermined problem on the sphere}

\author{Mouhamed Moustapha Fall}
\address{M. M. F.: African Institute for Mathematical Sciences in Senegal, KM 2, Route de
Joal, B.P. 14 18. Mbour, Senegal.}
\email{mouhamed.m.fall@aims-senegal.org}

\author{Ignace Aristide Minlend}
\address{I. A. M.:African Institute for Mathematical Sciences in Senegal, KM 2, Route de
Joal, B.P. 14 18. Mbour, Senegal. }
\email{ignace.a.minlend@aims-senegal.org}

\author{Tobias Weth}
\address{T.W.:  Goethe-Universit\"{a}t Frankfurt, Institut f\"{u}r Mathematik.
Robert-Mayer-Str. 10 D-60629 Frankfurt, Germany.}

\email{weth@math.uni-frankfurt.de}

\thanks{Work supported by DAAD and AvH}


\subjclass[2010]{Primary ; Secondary}

\begin{abstract}
We study Serrin's overdetermined
boundary value problem
\begin{equation*}
 -\Delta_{S^N}\, u=1 \quad \text{ in $\Omega$},\qquad u=0, \; \partial_\eta u=\textrm{const}  \quad \text{on $\partial \Omega$}
\end{equation*}
in subdomains $\Omega$ of the round unit sphere $S^N \subset
\R^{N+1}$, where $\Delta_{S^N}$ denotes the Laplace-Beltrami
operator on $S^N$. A subdomain $\Omega$ of $S^N$ is called a Serrin
domain if it admits a solution of this overdetermined problem.  In
our main result, we construct Serrin domains in $S^N$, $N \ge 2$
which bifurcate from symmetric straight tubular neighborhoods of the
equator. Our result provides the first example of Serrin domains in
$S^{N}$ which are not bounded by geodesic spheres.

\end{abstract}

\maketitle
\section{Introduction and main result}
The present paper is concerned with an overdetermined boundary value problem for the equation
\begin{equation}\label{eq:Pro-1}
 -\Delta_{S^N}  u=1 \qquad \text{ in $\Omega$}
\end{equation}
in a subdomain $\Omega$ of the unit sphere $S^{N}\subset \mathbb{R}^{N+1}$. More precisely, we are interested in domains $\Omega \subset S^N$ of class $C^2$ such that~(\ref{eq:Pro-1}) is solvable subject to the boundary conditions
\begin{equation}\label{eq:Pro-2}
u=0, \quad \partial_\eta u=\textrm{const}  \qquad \text{on $\partial
\Omega$.}
\end{equation}
Here $\Delta_{S^N}$ is the Laplace-Beltrami operator on $S^{N}$, and  $\eta$ is the unit outer normal on  $\partial \O$. A motivation to study this problem is given by classical and recent results for the corresponding problem in the euclidean setting, i.e. the problem of
finding a smooth domain $\O \subset \R^N$ such that the overdetermined problem
\begin{equation}\label{eq:Pro-1-eucl}
 -\Delta u=1 \quad \text{ in $\Omega$}, \qquad u=0,\quad \partial_\nu u=\textrm{const} \qquad \text{on $\partial \Omega$}
\end{equation}
admits a solution. The additional Neumann boundary condition in (\ref{eq:Pro-1-eucl}) arises in many applications as a shape optimization problem for the underlying domain $\Omega$. For a detailed discussion of some applications e.g. in fluid dynamics and the linear theory of torsion,
see \cite{Serrin,Sirakov}. The most celebrated result for the overdetermined problem (\ref{eq:Pro-1-eucl}) has been obtained by Serrin in 1971. In his paper  \cite{Serrin}, Serrin proved that, among smooth bounded domains $\O$, (\ref{eq:Pro-1-eucl}) only admits a solution if $\O$ is a ball. More precisely, in \cite{Serrin} the torsion equation $-\Delta u = 1$ is considered as an important  special case within a class of semilinear problems for which the same rigidity result is established. While Serrin's proof relies on the moving plane method which appeared for the first time in a PDE context in \cite{Serrin}, Weinberger \cite{Wein} found an alternative argument based on integral identities which yields the result in the special case of problem (\ref{eq:Pro-1-eucl}).

It is natural to ask if and in which sense Serrin's classification result carries over to the case of other ambient Riemannian manifolds $(\cM,g)$ in place of $(\R^N,g_{eucl})$. More precisely, letting $\Delta_\cM$ denote the Laplace-Beltrami operator on $\cM$, we are interested in rigidity and non-rigidity results for {\em Serrin domains in $\cM$}, i.e., for subdomains $\Omega \subset \cM$
in which the overdetermined problem
\begin{equation}
\label{eq:general-M}
 -\Delta_{\cM} u=1 \quad \text{ in $\Omega$},\qquad u=0, \quad \partial_\eta u=\textrm{const}  \quad \text{on $\partial \Omega$}
\end{equation}
is solvable. We note that in  \cite{FM-Serr}, the first two authors proved that Serrin domains, which are perturbation of small geodesic balls, exist in any compact Riemannian manifold. As observed by the second author in \cite{I.A.M}, Serrin domains arise in the context of Cheeger sets in a Riemannian framework. To explain this connection more precisely, we recall that the Cheeger constant of a Lipschitz subdomain $\Omega \subset \cM$ is given by
\be
\label{eq:def_Cheeg} h(\O):= \inf_{A\subset \O }\frac{P(A)}{|A|}.
\ee
Here the infimum is taken over measurable  subsets $A
\subset \Omega$, with finite perimeter $P(A)$ and where $|A|$  denotes the volume   of $A$  (both with respect to the metric $g$).  If $h(\Omega)$ is
attained by $\O$ itself, $\Omega$ is usually called self-Cheeger, and it is called uniquely self-Cheeger if $\Omega$ is the
only subset which attains $h(\Omega)$. By means of the P-function method, it is shown in \cite[Theorem 1.2]{I.A.M} that every Serrin domain in a compact Riemannian manifold $\cM$ with nonnegative Ricci curvature is uniquely self-Cheeger. Cheeger constants play an important role in eigenvalue estimates on Riemannian manifolds (see \cite{chavel}), whereas in the classical Euclidean case $(\cM,g)=(\R^N,g_{eucl})$ these notions have applications in the denoising  problem in  image processing, see e.g. \cite{EParini,Leon}.

The result of Serrin was extended in \cite{Ku-Pra} to subdomains $\Omega$ of the round sphere $\cM = S^{N}$ which are contained in a hemisphere.
 More precisely, it is proved in \cite{Ku-Pra} that any smooth Serrin domain $\Omega$ contained in a hemisphere of $S^N$ is a geodesic ball.

However, geodesic balls are not the only Serrin domains in $S^N$. Identifying $S^{N-1}$ with the equator $S^{N-1} \times \{0\} \subset S^{N}$,
it is easy to see that the symmetric straight tubular neighborhoods
\begin{equation}
  \label{eq:straight-tubular-neighborhoods}
\Bigl\{\bigl((\cos \theta) \sigma, \sin  \theta\bigr) \::\:
\sigma \in S^{N-1},\: \:|\theta|<\lambda\Bigr\} \subset S^N, \qquad
0<\lambda < \frac{\pi}{2}
\end{equation}
of $S^{N-1}$ are Serrin domains, see Section~\ref{sec:transf-probl-its} below for details.
We note that, as the geodesic balls, these domains are also bounded by geodesic spheres in $S^{N}$.
It is therefore tempting to guess that this property is shared by any Serrin domain in $S^N$.
In the present paper we show that this is not the case.
More precisely,  we construct Serrin domains in $S^{N}$ of the form
\begin{equation}
  \label{eq:def-D-phi}
D_\phi:= \Bigl\{\bigl((\cos \theta) \sigma
         , \sin \theta \bigr) \::\: \sigma \in S^{N-1},\: \theta \in \R,\:
|\theta| < \phi(\sigma)\Bigr\} \subset S^N,
\end{equation}
where, in contrast to the sets given in (\ref{eq:straight-tubular-neighborhoods}),
$\phi: S^{N-1} \to \R$ is a nonconstant $C^2$-function  with $0< \phi < \frac{\pi}{2}$.

To state our main result, we need to introduce some notions. In the following, we call a non constant  function $\phi: S^{N-1} \to \R$ {\em axially symmetric (with respect to the axis $\R e_1$)} if
\begin{equation}
  \label{eq:def-axial-symmetry}
\phi \equiv \phi \circ A \qquad \text{for any $A \in O(N)$ with $A e_1 = e_1$.}
\end{equation}
Moreover, for a nonnegative integer $j$, we let $Y_j$ be the unique $L^2$-normalized real-valued axially
symmetric spherical harmonic of degree $j$. So $Y_j$ is the
 restriction of the uniquely determined real-valued harmonic polynomial $P_j$ of degree $j$ to $S^{N-1}$
  which  satisfies (\ref{eq:def-axial-symmetry}) and
\begin{equation*}
\int_{S^{N-1}} P_j^2\,d\sigma =\int_{S^{N-1}} Y_j^2\,d\sigma=1.
\end{equation*}
Here and in the following, $d\sigma$ always denotes the volume element of $S^{N-1}$. The unique existence of $Y_j$ is well known; a proof and some explicit formulas for $Y_j$ are given e.g. in
\cite[Chapter 2]{Helms}. Our main result is the following.
\begin{Theorem}\label{teo1}
Let $N\geq 2$ and $\alpha \in (0,1)$. Then there exists a strictly decreasing sequence $(\lambda_j)_{j \ge 2}$ with $\lim \limits_{j \to \infty}\lambda_j = 0$
and the following property:\\
For each $j  \ge 2$, there exists $\eps_j>0$ and a curve
\begin{align*}
 (-\e_j,\e_j) &\longrightarrow  (0,\infty) \times C^{2,\alpha}(S^{N-1})\\
        s&\longmapsto (\lambda_j(s),\vp_s^j)
 \end{align*}
with $\vp_0^j\equiv 0$, $\lambda_j(0)=\lambda_j$ and such that for all
$s\in(-\e_j,\e_j),$ letting $\phi^{j}_s=\l_j(s)+\vp_s^j$,  there
exists a solution $u \in C^{2,\a}(\ov{D_{\phi^{j}_s}})$ of the
overdetermined problem
\begin{equation}
 \label{eq:Proe1}
 \left\{
   \begin{aligned}
  -\Delta_{\text{\tiny $S^N$}}\, u&=1&&\qquad \textrm{in}\quad D_{\phi^{j}_s}\\
u&=0,\quad  \partial_\nu u= const &&\qquad \textrm{on }\quad
\partial D_{\phi^{j}_s}.
 \end{aligned}
\right.
\end{equation}
Moreover,  for every $s \in (-\e_j,\e_j)$, the function $\vp_s^j \in C^{2,\alpha}(S^{N-1})$ is axially symmetric. Furthermore, we
have
$$
\vp_s^j=s\Bigl(Y_j +w_s^j\Bigr) \qquad \text{in $C^{2,\alpha}(S^{N-1})$ for $s \in (-\e_j,\e_j)$,}
$$
where $s\mapsto w_s^j$ is a smooth map
$(-\e_j,\e_j) \to  C^{2,\alpha}(S^N)$ satisfying
$w_0^j\equiv 0$ and
\begin{equation}
  \label{eq:orthogonality-cond-main-thm}
\int_{S^{N-1}} w_s^j(\sigma) Y_j(\sigma) \,d\sigma = 0 \qquad
\text{for $s \in (-\e_j,\e_j)$.}
\end{equation}
\end{Theorem}
%For $N=2$ the domains $D_{\phi^{j}_s}$ look as in Figure 1.
\begin{figure}[H]
 %  \begin{minipage}[c]{.46\linewidth}
 %     \includegraphics[scale=0.4]{S13.PNG}
 %  \end{minipage} \hfill
 %  \begin{minipage}[c]{.46\linewidth}
 %     \includegraphics[scale=0.4]{S11.PNG}
 %  \end{minipage}
   \includegraphics[scale=0.4]{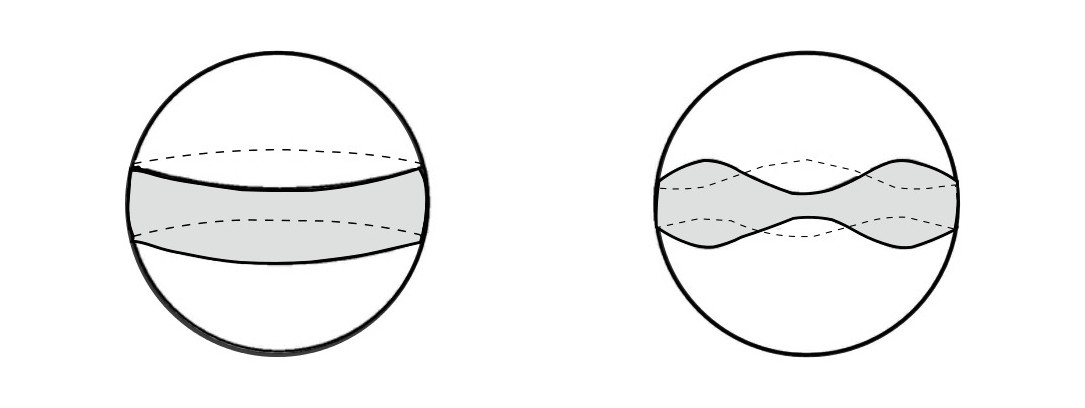}
   \caption{The domains  $D_{\l_j}$ on the left and $D_{\phi^{j}_s}$, for $j=3$, $s \not =0$ on the right.}\label{fid:S-d}
\end{figure}
We note that the orthogonality condition
(\ref{eq:orthogonality-cond-main-thm}) implies that the functions
$\phi^{j}_s$ are not constant for $s \not = 0$, so the corresponding
domains $D_{\phi^{j}_s}$ are nontrivial, as illustrated in Figure \ref{fid:S-d}. We also note that the axial
symmetry of the functions $\vp_s^j$ in Theorem~\ref{teo1} leads to
domains $D_{\phi^{j}_s} \subset S^N$ having an axial (or radial)
symmetry of codimension two, i.e., they are invariant under 
every orthogonal transformation which leaves the two-dimensional
subspace $\R \times \{0_{\,\R^{N-1}}\} \times \R \subset \R^{N+1}$
fixed. We recall here that the unit vector $e_1$ lies in the same
hyperplane as $S^{N-1} \subset S^N$. In the case $N=2$, this
property merely amounts to reflection symmetry of the domains
$D_{\phi^{j}_s}$ with respect to the two-dimensional hyperplane $\R
\times \{0\} \times \R \subset \R^{3}$.

Theorem~\ref{teo1} supports the perception that the structure of the set of Serrin domains in a manifold $(\cM,g)$ is similar to the structure of closed embedded constant mean curvature hypersurfaces (CMC hypersurfaces in short) in $\cM$ up to a shift of the dimension. We recall that the rigity result in \cite{Serrin} for Serrin domains in $\R^N$ parallels an earlier important result by Alexandrov \cite{Alexandrov} stating that closed embedded $CMC$-hypersurfaces in $\R^N$ are round spheres. The result in \cite{Alexandrov} is based on the moving plane method in geometric form which is also refered to as Alexandrov's reflection principle. By the same technique, Alexandrov proved in \cite{Alexandrov-62} that any closed embedded $CMC$-hypersurface contained in a hemisphere of $S^N$ is a geodesic sphere. An explicit family of embedded $CMC$-hypersurfaces in $S^N$ with nonconstant principal curvatures has been found by Perdomo in \cite{Perdomo} in the case $N \ge 3$. These hypersurfaces seem somewhat related to the Serrin domains given by Theorem~\ref{teo1}, although they bound a tubular neighborhood of $S^1$ and not of $S^{N-1}$. We also emphasize that Theorem~\ref{teo1} applies in the case of $S^2$, whereas $CMC$-hypersurfaces in $S^2$ are obviously trivial, i.e., they are geodesic circles.

A similar dimensional shift is present in the case where $\cM= T^m \times \R^k$ with the flat metric, where $T^m = \R^m / 2 \pi \Z^m$ denotes the $m$-torus. In the recent paper \cite{Fall-Minlend-Weth}, the authors have constructed bounded Serrin domains in $\cM$ with nonconstant principal curvatures for arbitrary $m \ge 1$, $k \ge 1$. These domains clearly correspond to periodic unbounded domains in $\R^{m} \times \R^k$, and they should be seen as analogues of the classical one-parameter family of Delaunay hypersurfaces \cite{Delau} in the case $m=1$, $k=2$. Again we stress that there is no analogue in the two-dimensional case $m=k=1$; every bounded CMC-hypersurface in $S^1 \times \R$ is clearly trivial, i.e., it is of the form $S^1 \times \{t\}$ with $t \in \R$. On the other hand, the construction of Delaunay hypersurfaces has been generalized to higher dimensions in \cite{Hsiang-Yu}.

The richer structure of the set of Serrin domains in low dimensions can be understood as a consequence of the nonlocal nature of geometric optimization problems for overdetermined boundary value problems. As it will become clear in Section~\ref{sec:transf-probl-its} below, these problems naturally lead to the study of (nonlocal) Dirichlet-to-Neumann maps.

This phenomenon has already been studied in detail for other classes of overdetermined boundary value problems involving different elliptic equations. Starting with the pioneering papers of Sicbaldi\cite{Sic} and Hauswirth, H\'elein and Pacard \cite{hauswirth-et-al}, the construction of nontrivial domains giving rise to solutions of overdetermined problems has been performed in many specific settings, see e.g. \cite{del-Pino-pacard-wei,ScSi,Ros-Ruiz-Sicbaldi-2016,morabito-sicbaldi}. Moreover, rigidity results for these domains were derived in \cite{farina-valdinoci,farina-valdinoci:2010-1,farina-valdinoci:2010-2,farina-valdinoci:2013-1,Ros-Sicbaldi,Ros-Ruiz-Sicbaldi-2015,Traizet}.

We note that most of the work so far has been devoted to the
Euclidean setting $(\cM,g)=(\R^N,g_{eucl})$. In the very interesting
recent paper \cite{morabito-sicbaldi}, the more closely related case
$\cM= S^{N-1} \times \R$ is considered for the equation
$-\Delta_{\cM}u = \lambda u$, but the domains considered there are
unbounded with respect to $\R$-variable and thus very different from
subdomains of $S^N$. Note also that the Serrin domains obtained in  \cite{FM-Serr} are trivially geodesic spheres  by the rigidity result in  \cite{Ku-Pra}, since these domains are obtained by perturbing geodesic balls with small radius. In fact we are not aware of any other
construction of nontrivial subdomains of the sphere $S^N$ arising in
the context of overdetermined boundary value problems.

 Our approach for proving Theorem \ref{teo1} relies on the Crandall-Rabinowitz Bifurcation Theorem \cite[Theorem 1.7]{M.CR}.
 The main task is to reduce the problem to an operator equation in appropriate function spaces
 such that an associated family of linearized operators satisfy the spectral and Fredholm type transversality assumptions of this theorem.
 In particular, we need to derive precise information on the eigenvalues and eigenfunctions of related families of ODEs.

We explain the main steps while describing the organization of the paper. In Section~\ref{sec:transf-probl-its},

 we consider domains of the form $D_{\phi}$ given in (\ref{eq:def-D-phi}) with $\phi \in C^{2,\alpha}(S^{N-1})$, $0< \phi< \frac{\pi}{2}$
 and transform the corresponding overdetermined problem to an equivalent overdetermined boundary value
 problem on a fixed domain with a $\phi$-dependent metric.
 The solvability condition for this problem is then formulated as an operator equation in $\phi$ of the form $H(\phi) \equiv const$,
where $H: C^{2,\alpha}(S^{N-1}) \to C^{1,\alpha}(S^{N-1})$ is a nonlinear operator of Dirichlet-to-Neumann type.
We then compute its linearizations $\mbL_\lambda:=D H(\l)$ at constant functions $\phi \equiv \lambda$, $\lambda \in (0,\frac{\pi}{2})$
in the form of a classical Dirichlet-to-Neumann operator, see Proposition~\ref{sec:peri-solut-serr-2}. In Section~\ref{proposition-eigenvalues},
we then study the eigenvalues $\sigma_j(\lambda)$, $j \in \N \cup \{0\}$ and corresponding eigenfunctions of $\mbL_\lambda$ given by the spherical
harmonics of degree $j$. By a separation of variables ansatz using spherical harmonics, we represent these eigenvalues via boundary derivatives
of the solutions of an associated $\lambda$- and $j$-dependent family of second order ODEs, see (\ref{eq:problem4}) below.

However, this representation is not too useful to derive detailed qualitative information.
In Proposition~\ref{proposition-eigenvalues}, we therefore derive a different formula for the values $\sigma_j(\lambda)$
which only involves a $j$-dependent family of first order ODEs {\em in the parameter $\lambda$}.
This  somewhat surprising new role of the parameter $\lambda$ might be a consequence of a more general principle
that we are unaware of so far. In Section~\ref{sec:function-spaces}, we use the information on the eigenvalues and
eigenfunctions to deduce mapping properties of $\mbL_\lambda$ in Sobolev spaces.

In Section~\ref{sec:compl-proof-theor}, we then complete the proof of Theorem~\ref{teo1}
  via the Crandall-Rabinowitz Bifurcation Theorem.
  For this we need to restrict the operator $H$ and its linearizations $\mbL_\lambda$ to the subspace
  of axially symmetric functions in $C^{2,\alpha}(S^{N-1})$,
  because the intersections of this subspace with the eigenspaces of $\mbL_\lambda$ are one-dimensional, as required in \cite[Theorem 1.7]{M.CR}. We also note that, in order to ensure the transversality condition in \cite[Theorem 1.7]{M.CR}, we are led to show that $\frac{d\sigma_j(\lambda)}{d\lambda}>0$ for $j \ge 2$ and $\lambda \in (0,\frac{\pi}{2})$ with $\sigma_j(\lambda)=0$, see Proposition~\ref{eq:OK-egesff} in Section~\ref{sec:spectr-prop-line-1}.

Finally, we emphasize that Theorem~\ref{teo1} yields bifurcation of Serrin domains only at simple eigenvalues  $\s_j(\l_j)$ corresponding to axially symmetric spherical harmonics $Y_j$ of higher degree $j\geq 2$. Such a bifurcation cannot happen for $j=1$.  Indeed, as we shall see  in  Proposition \ref{proposition-eigenvalues} below,  $\s_1(\l)<0$, for $\l\in (0,\frac{\pi}{2})$, while this eigenvalue corresponds to the (nonconstant) axially symmetric eigenfunction $x\mapsto Y_1(x)=x_1$.
This contrasts our paper \cite{Fall-Minlend-Weth} devoted to $\cM= T^m \times \R^k$, where bifurcation of Serrin's domain occurs at the every simple eigenvalue corresponding to  a nonconstant eigenfunction.\\

\textbf{Acknowledgement:}
M.M.F. is supported by the Alexander von Humboldt Foundation. Part of the paper was written while I.A.M. and M.M F. visited the Institute of Mathematics of the Goethe-University Frankfurt. They wish to thank the institute for its hospitality and the German Academic Exchange Service (DAAD) for funding of the visit of I.A.M. within the program 57060778. Also T.W. wishes to thank DAAD for funding within the program 57060778.
The authors also thank Wolfgang Reichel for drawing their attention to overdetermined problems on the $N$-sphere.\\
The authors would like to thank the referee for carefully reading the paper and for his/her valuable comments.

\section{The transformed problem and its linearization}
\label{sec:transf-probl-its}

Consider the $N$-sphere $S^{N}\subset\mathbb{R}^{N+1}$. The open subset $S^N \setminus \{\pm e_{N+1}\}$ is locally parameterized by
\begin{align*}
 \Upsilon:  S^{N-1}\times (-\frac{\pi}{2},\frac{\pi}{2})&\longrightarrow
S^{N},\quad (\sigma,\theta)\mapsto \bigl((\cos \theta)\sigma, \sin
\theta \bigr).
 \end{align*}
The standard metric of  $S^{N}$ in coordinates $(\sigma,\theta)$ is given by
\begin{equation}\label{eq:me}
g_{(\sigma,\theta)}=\cos^2 \theta \,
g_{S^{N-1}}(\sigma)+\textrm{d}\theta^2,
\end{equation}
where $g_{S^{N-1}}$ denotes the standard metric on
$S^{N-1}$. In these coordinates, the corresponding Laplace-Beltrami operator of $S^{N}$ writes as
\begin{equation}\label{eq:lapla}
\Delta_{g_{(\sigma,\theta)}}=\frac{1}{\cos^{N-1} \theta}\: \frac{\partial}{\partial
\theta}\biggl(\cos^{N-1} \theta \:\frac{\partial }{\partial
\theta}\biggl)+\frac{1}{\cos^{2} \theta}\Delta_{S^{N-1}}.
\end{equation}
We fix $\alpha \in (0,1)$ in the following.  For $j \in \N \cup \{0\}$,  we consider the Banach space $C^{j,\alpha}(S^{N-1})$, and we let
$$
\cU:= \{ \phi \in C^{2,\alpha}(S^{N-1}) \::\: 0< \phi < \frac{\pi}{2}
\}.
$$
For a function $\phi \in \cU$, we define
$$
\O_\phi:=\left\{ (\sigma,\theta ) \in S^{N-1} \times  (-\frac{\pi}{2},\frac{\pi}{2}) \,:\,
|\theta| < \phi(\sigma)  \right\}.
$$
We note  that $ \Upsilon(\O_\phi)=D_\phi,$ with $D_\phi$ defined in (\ref{eq:def-D-phi}). We are therefore led to consider the problem
\begin{equation}\label{eq:Proe}
\left\{
 \begin{aligned}
-\Delta_{g_{(\sigma,\theta)}} u&= 1 &&\qquad \textrm{in $\O_\phi$},\\
u&=0  &&\qquad \textrm{on $\partial \O_\phi.$}
  \end{aligned}
\right.
  \end{equation}
We are looking for $\phi \in \cU$ such that the unique solution of
\eqref{eq:Proe} satisfies the additional boundary condition
\begin{equation}
\label{eq:Proe1-1-neumann-0}
\partial_{\mu_\phi} u= -c \quad\textrm{ on }\quad \partial \O_\phi,
\end{equation}
where $\mu_\phi$ is the outer unit normal on $\partial \O_\phi$ with respect to the metric $g_{(\sigma,\theta)}$, which is given by
\begin{equation}
  \label{eq:eta-normal}
\mu_\phi(\sigma,\theta) = \frac{(-\frac{\nabla \phi(\sigma)}{\cos \theta}, \frac{\theta}{|\theta|})}{\sqrt{1+ |\nabla \phi(\sigma)|^2}} \;\in \: T_{\sigma}S^{N-1} \times \R \qquad \text{for $(\sigma,\theta ) \in \partial \O_\phi$.}
\end{equation}
Here and in the following, $\nabla \phi(\sigma) \in T_{\sigma}S^{N-1}$ denotes the gradient with respect to $g_{S^{N-1}}$.

In the following, we wish to pull back the problem~(\ref{eq:Proe}) to a problem on a fixed domain. For this we note that every $\phi \in \cU$ gives rise to a  $C^{2,\alpha}$-regular map
$$
\Psi_\phi: S^{N-1} \times \mathbb{R}  \to S^{N-1} \times \mathbb{R} ,
\quad \qquad \Psi_\phi(\sigma, t)= (\sigma , \phi(\sigma)t)
$$
such
that $\Psi_\phi$ maps $$\O:=S^{N-1} \times (-1,1)$$ diffeomorphically
onto $\O_{\phi}.$  Let the metric $g_\phi$ be defined as the pull
back of the metric $g_{(\sigma,\theta)}$ under
the map $\Psi_\phi$, so that $\Psi_{\phi}:(\overline \O,g_\phi) \to
(\overline \O_{\phi},g_{(\sigma,\theta)})$ is an isometry.  Then our
problem is equivalent to the overdetermined problem consisting of
\begin{equation}\label{eq:Proe1-1}
\left\{
 \begin{aligned}
    -\Delta_{g_\phi} u&= 1&&\qquad \textrm{in $\O$,}\\
u&=0&&\qquad \textrm{on $\partial \O$}\\
  \end{aligned}
\right.
  \end{equation}
and the additional Neumann condition
\begin{equation}
\label{eq:Proe1-1-neumann}
\partial_{\nu_{\phi}} u= -c \quad\textrm{ on }\quad \partial \O.
\end{equation}
Here
$$
\nu_\phi: \partial \O \to \R^{N+1}
$$
is the unit outer normal at $\partial \O \subset S^{N-1} \times \R \subset \R^{N+1}$ with respect to $g_\phi$.
Since $\Psi_\phi: (\overline \Omega, g_\phi) \to (\overline {\Omega_\phi}, g_{(\sigma,\theta)})$ is an isometry, we have
\begin{equation}
  \label{eq:rel-eta-phi-nu-phi}
[d \Psi_\phi ]\nu_{\phi} =  \mu_\phi \circ \Psi_\phi  \qquad \text{on $\partial \Omega$}
\end{equation}
with $\mu_\phi$ is given in (\ref{eq:eta-normal}).

Here and in the following, we distinguish different types
of derivatives in our notation. If $f: O \to \R^\ell$ is a $C^1$-map defined on an open set $O \subset S^{N-1} \times \R$,
we write $d f(\sigma,t) \in \cL(T_\sigma S^{N-1} \times \R, \R^\ell)$ for its derivative at a point  $(\sigma,t) \in O$.
In contrast, we shall use the symbols $D$ or $D_\phi$ to denote functional derivatives. More precisely,
if $X,Y$ are infinite dimensional normed (function) spaces and $F \in C^1(O, Y)$, where $O \subset X$ is open,
we let $D F(\phi)$ or $D_\phi F(\phi) \in \cL(X,Y)$ denote the Fr\'echet derivative of $F$ at a function $\phi \in O$.

The following lemma is concerned with the well-posedness of problem \eqref{eq:Proe1-1}.
\begin{Lemma}
\label{sec:peri-solut-serr-1} For any $\phi \in \cU$, there is a
unique solution $u_\phi \in C^{2,\alpha}(\overline \O)$ of
\eqref{eq:Proe1-1}, and the map
\begin{equation}
  \label{eq:def-smooth-map}
\cU \to C^{2,\alpha}(\overline \O),\qquad \phi \mapsto u_\phi
\end{equation}
is smooth. Moreover we have the following properties.
\begin{itemize}
\item[(i)] For any $\phi \in \cU$, the functions $u_\phi: \overline \Omega \to \R$ and $\partial_{\nu_\phi} u_\phi: \partial \O \to \R$ are even in the variable  $t \in (-1,1)$.
\item[(ii)] If $\phi= \phi(\sigma,t)$ is axially symmetric in $\sigma \in S^{N-1}$, then the functions $u_\phi: \overline \Omega \to \R$ and $\partial_{\nu_\phi} u_\phi: \partial \O \to \R$ are axially symmetric in $\sigma$ as well.
\end{itemize}
\end{Lemma}

\proof
Let
$$
X:= \{u \in C^{2,\alpha}(\overline \O)\::\: u = 0 \;\text{on $\partial \O$}\}\qquad \text{and}\qquad Y= C^{0,\alpha}(\overline \O).
$$
Moreover, let $\cL(X,Y)$ denote the space of bounded linear operators $X \to Y$, and let $\cI(X,Y) \subset \cL(X,Y)$ denote the set of topological isomorphisms $X \to Y$.  Since the metric coefficients of $g_\phi$ are smooth functions of $\phi$ and $d \phi$, it is easy to see that the map
$$
m: \cU \to \cL(X,Y), \qquad \phi \to m(\phi):= -\Delta_{g_\phi}
$$
is smooth. Moreover, for $\phi \in \cU$, the definition of $g_\phi$ implies that $\Delta_{g_\phi}$ is an elliptic, coercive second order differential operator in divergence form with $C^{1,\alpha}(\overline \O)$-coefficients. This readily implies, by the maximum principle and elliptic regularity, that $m(\phi) \in \cI(X,Y)$ for every $\phi \in \cU$, and consequently the problem (\ref{eq:Proe1-1}) has a unique solution $u_\phi \in X$ for every $\phi \in \cU$. We now recall that $\cI(X,Y) \subset \cL(X,Y)$ is an open set and that the inversion
$$
{\rm inv}: \cI(X,Y) \to \cI(Y, X), \quad {\rm inv}(A)= A^{-1}
$$
is smooth. Since $u_\phi= \inv(m(\phi)) 1$, the smoothness of the map in (\ref{eq:def-smooth-map}) follows.\\
Finally, properties (i) and (ii) follows from the uniqueness of $u_\phi$.
\QED

By Lemma~\ref{sec:peri-solut-serr-1}, condition (\ref{eq:Proe1-1-neumann}) is equivalent to
\begin{equation}
  \label{eq:reformulated-overd}
[\partial_{\nu_{\phi}} u_{\phi}](\sigma,1)=-c \qquad \text{for $\sigma  \in S^{N-1}$.}
\end{equation}
It follows from (\ref{eq:eta-normal}) and (\ref{eq:rel-eta-phi-nu-phi}) that the map
$$
\cU  \to C^{1,\alpha}(\partial \O,\R^{N+1}), \qquad \phi \mapsto \nu_\phi
$$
is smooth, and thus we have a smooth map
\begin{equation}
  \label{eq:def-H}
H: \cU \to C^{1,\alpha}(S^{N-1}), \qquad H(\phi)(\sigma)= \partial_{\nu_{\phi}} u_{\phi} (\sigma,1).
\end{equation}
Moreover, (\ref{eq:reformulated-overd}) writes as
\begin{equation}
  \label{eq:reformulated-overd-H}
H(\phi)\equiv -c  \qquad \text{on $S^{N-1}$}.
\end{equation}
In order to find solutions of the latter equation bifurcating from the trivial branch of
solutions $\phi \equiv \lambda$, $\lambda>0$, we need to study the linearization of $H$ at constant functions. In the next lemma, we collect the form of various quantities in the case $\phi \equiv \lambda$ for later use.

\begin{Lemma}
\label{eq: phi=lambda-lemma}
In the case where $\phi \equiv \lambda \in (0,\frac{\pi}{2})$ is a constant, we have
\begin{align*}
\Omega_\lambda &= S^{N-1} \times ]-\lambda,\lambda[,\\
\mu_\lambda(\sigma,\theta)&= (0,\text{\rm sgn} \,\theta)\; \in T_\sigma S^{N-1} \times \R \qquad \text{for $(\sigma,\theta) \in \partial \Omega_\lambda$,}\\
\Psi_\lambda(\sigma,t)&= (\sigma,\lambda t)\qquad \qquad \qquad \qquad \qquad \;\text{for $(\sigma,t) \in S^{N-1} \times \R$},\\
g_{\lambda}(\sigma,t)&=\lambda^{2}\text{\rm d}t^{2}+\cos^{2}(\lambda
t)g_{S^{N-1}}(\sigma) \qquad \; \text{for $(\sigma,t) \in \overline{\Omega}$},\\
\nu_\lambda(\sigma,t)&= (0,\frac{\text{\rm sgn}(t)}{\lambda})\; \in T_\sigma S^{N-1} \times \R \quad \;\; \;\text{for $(\sigma,t) \in \partial \Omega$  }
\end{align*}
and
\begin{equation}
  \label{eq:laplace-lambdaa}
\Delta_{g_\lambda}  =  \frac{1}{\lambda^{2}\cos^{N-1} (\lambda t)} \partial_t \Bigl( \cos^{N-1}(\lambda t) \partial_t\Bigr)  + \frac{1}{\cos^2 (\lambda t)}\Delta_{S^{N-1}},
\end{equation}
where $\Delta_{S^{N-1}}$ acts only on the $\sigma$-variable. Moreover, we have
\begin{equation}
  \label{eq:form-u-lambda}
 u_\lambda(\sigma,t)= \tilde u(\lambda t)-\tilde u(\lambda) \qquad  \text{for $(\sigma,t) \in \overline{\Omega}$},\\
\end{equation}
where $\tilde u: (-\frac{\pi}{2},\frac{\pi}{2}) \to \R$ is the even solution of the ODE
\begin{equation}\label{eq: lapalsolini}
\frac{1}{\cos^{N-1}
\theta}\frac{\partial}{\partial\theta}\Bigl( \cos^{N-1}
\theta\frac{\partial}{\partial\theta} \Bigr)\tilde u(\theta)
= -1, \qquad \text{$\theta \in
(-\frac{\pi}{2},\frac{\pi}{2})$,}
\end{equation}
which is unique up to an additive constant. Finally, we have
\begin{equation}
  \label{calcH-lambda}
 H(\lambda) = \tilde u'(\lambda).
\end{equation}
\end{Lemma}

\proof
The expressions for $\Omega_\lambda,\mu_\lambda, \Psi_\lambda$, $g_\lambda$ and $\nu_\lambda$ are immediate. To prove (\ref{eq:laplace-lambdaa}), we note that
$$
|g_{\lambda}(\sigma,t)|=\lambda^{2}\cos^{2(N-1)}(\lambda t)|g_{S^{N-1}}|
$$
and
\begin{equation}\label{matrix1}
g^{-1}_{\lambda}(\sigma,t)=
 \begin{pmatrix}
  \cos^{-2}(\lambda t)g^{-1}_{S^{N-1}}(\sigma) &  0 \\
0  &\lambda^{-2}
\end{pmatrix}
\qquad \text{for $(\sigma,t) \in \overline{\Omega}$},
\end{equation}
which implies that
\begin{align*}
\Delta_{g_\lambda}=&\frac{1}{\sqrt{|g_{\lambda}|}}\partial_{i}\left(\sqrt{|g_{\lambda}|}g^{ij}_{\lambda}\partial_j\right)=\frac{1}{\lambda\cos^{N-1}(\lambda
t)\sqrt{|g_{S^{N-1}}|}} \\
& \cdot \Bigl[ \partial_{t}\Bigl( \frac{\cos^{N-1}(\lambda
t)}{\lambda} \sqrt{|g_{S^{N-1}}|} \partial_t\Bigr) +
\partial_{i}\Bigl( \lambda \cos^{N-3}(\lambda
t)\sqrt{|g_{S^{N-1}}|} g^{ij}_{S^{N-1}}(\sigma)\partial_j\Bigr)\Bigr]\\
=& \frac{1}{\lambda^{2}\cos^{N-1} (\lambda t)}
\partial_t \Bigl( \cos^{N-1}(\lambda t) \partial_t\Bigr) +\frac{1}{\cos^2(\lambda t) \sqrt{|g_{S^{N-1}}|}}\partial_{i}\biggl( \sqrt{|g_{S^{N-1}}|}g^{ij}_{S^{N-1}}(\sigma)\partial_j\biggl)\\
=&\frac{1}{\lambda^{2}\cos^{N-1} (\lambda t)}
\partial_t \Bigl( \cos^{N-1}(\lambda t) \partial_t\Bigr)  +
\frac{1}{\cos^2 (\lambda t)}\Delta_{S^{N-1}},
\end{align*}
as claimed in (\ref{eq:laplace-lambdaa}).

Next, let $u_\lambda$ be given as in  \eqref{eq:form-u-lambda}. Then, as a consequence of (\ref{eq:laplace-lambdaa}), $u_\lambda$ solves the Dirichlet problem (\ref{eq:Proe1-1}) with $\phi \equiv \lambda$, and this shows that the formula in \eqref{eq:form-u-lambda} is correct. Finally, we have
$$
 H(\lambda) = \partial_{\nu_\lambda}u_\lambda = \tilde u'(\lambda)
$$
as a consequence of (\ref{eq:form-u-lambda}), the fact that $\partial_{\nu_\lambda} =  \frac{\text{\rm sgn}(t)}{\lambda}\partial_t$ and since $\tilde{u}$ is even.
\QED

We have the following estimate for the first  derivative of the even solution $\tilde u$ of \eqref{eq: lapalsolini}.
\begin{Lemma}\label{eq: solappro}
We have
\begin{equation}\label{ine der}
-\tilde{u}'(\theta)\geq \theta  \quad \textrm{for $\theta\in (0,\frac{\pi}{2})$. }
\end{equation}
\end{Lemma}
\proof
Since $\tilde u$ is even, we have $\tilde u'(0)=0$ and therefore
$$
\cos^{N-1}
\theta \: \tilde u'(\theta)= -\int_0^\theta \cos^{N-1} \vartheta \: d \vartheta \qquad \text{for $\theta \in (-\frac{\pi}{2},\frac{\pi}{2})$}
$$
by \eqref{eq: lapalsolini}. Consequently,
$$
-\tilde u'(\theta)= \int_0^\theta \frac{\cos^{N-1} \vartheta }{\cos^{N-1} \theta }d\vartheta \ge  \theta \qquad \text{for $\theta \in (0,\frac{\pi}{2})$.}
$$
 \QED

 The following is the main result of this section.
\begin{Proposition}
\label{sec:peri-solut-serr-2} Let $H: \cU \to C^{1,\alpha}(S^{N-1})$ be the smooth map defined  in (\ref{eq:def-H}). At a constant function $\lambda \in
(0,\frac{\pi}{2})$, the operator $\mbL_\lambda : = D H(\lambda) \in
\cL(C^{2,\alpha}(S^{N-1}), C^{1,\alpha}(S^{N-1}))$ is given by
\begin{equation}
  \label{eq:diff-H-express}
[\mbL_\lambda \omega] (\sigma) = -\frac{\tilde u'(\lambda)}{\lambda}
\partial_{t}\, \psi_{\omega,\lambda}(\sigma,1)+\tilde u''(\lambda)
\omega(\sigma) \qquad \text{for $\omega \in C^{2,\alpha}(S^{N-1})$,
$\sigma \in S^{N-1}$,}
\end{equation}
where $\psi_{\omega,\lambda} \in C^{2,\alpha}(\overline \O)$ is the unique solution of the problem
\begin{equation}
  \label{eq:psi-omega}
\left \{
  \begin{aligned}
 &\Delta_{g_\lambda} \psi_{\omega,\lambda} =0 && \qquad \text{in  $\O$},\\
&\psi_{\omega,\lambda}(\sigma,t)= \omega(\sigma)  && \qquad \text{for $(\sigma,t) \in \partial \O$,}
  \end{aligned}
\right.
\end{equation}
and
\begin{equation}
  \label{eq:laplace-lambda}
\Delta_{g_\lambda}  =  \frac{1}{\lambda^{2}\cos^{N-1} (\lambda t)} \partial_t \Bigl( \cos^{N-1}(\lambda t) \partial_t\Bigr)  + \frac{1}{\cos^2 (\lambda t)}\Delta_{S^{N-1}}.
\end{equation}
 Here $\Delta_{S^{N-1}}$ acts only on the $\sigma$-variable.
\end{Proposition}

The remainder of this section is devoted to the proof of Proposition~\ref{sec:peri-solut-serr-2}. To abbreviate the notation,  we put
$$
\tilde \nu_\phi(\omega):= [D_\phi \nu_\phi]\omega \in C^{1,\alpha}(\partial \O,\R^{N+1}) \qquad \text{for $\phi \in \cU,\: \omega \in C^{2,\alpha}(S^{N-1})$.}
$$
We start with the following simple observations.

\begin{Lemma}$ $\\
\label{sec:peri-solut-serr-3}
\begin{itemize}
\item[(i)] The map
$$
\cU \to C^{2,\alpha}(\overline \O, S^{N-1} \times \R) \subset C^{2,\alpha}(\overline \O, \R^{N+1}),\qquad \phi \mapsto \Psi_\phi
$$
is smooth. Moreover, for $\omega \in C^{2,\alpha}(S^{N-1})$,   we have
\begin{equation}
  \label{eq:diffPsi}
[D_\phi \Psi_{\phi}]\omega(\sigma,t) =  (0,\omega(\sigma)t) \qquad
\text{for $(\sigma,t) \in \overline \O$.}
\end{equation}
\item[(ii)]  Let
$$
m: \cU \to C^{2,\alpha}(\overline \O), \qquad \phi \mapsto m_\phi
$$
be a smooth map. Then the map
$$
\cM: \cU \to C^{1,\alpha}(\partial \O), \qquad \cM(\phi)=
\partial_{\nu_{\phi}} m_\phi
$$
is smooth as well and satisfies
$$
D_\phi \cM(\phi)\omega = \partial_{\tilde \nu_\phi (\omega)} m_\phi +  \partial_{\nu_{\phi}} \Bigl([D_{\phi} m_\phi]\omega\Bigr) \qquad \text{for $\omega \in C^{2,\alpha}(S^{N-1})$.}
$$
\end{itemize}
\end{Lemma}

\proof
(i) follows immediately from the definition of $\Psi_\phi$.\\
(ii) For $\phi \in \cU$ and $(\sigma,t) \in \partial \Omega$, we have
$$
\cM(\phi)(\sigma,t) = \partial_{\nu_\phi} m_\phi (\sigma,t)= D
m_\phi(\sigma,t) \nu_{\phi}(\sigma,t).
$$
Hence $\cM$ is smooth as a bilinear form composed with smooth functions, and $D \cM$ is given by
$$
D_\phi \cM(\phi)\omega = D  [(D_\phi m_\phi)\omega]  \nu_{\phi} + D m_\phi   \tilde \nu_\phi (\omega) = \partial_{\nu_{\phi}} \Bigl([D_{\phi} m_\phi]\omega\Bigr)
+ \partial_{\tilde \nu_\phi (\omega)} m_\phi.
$$
\QED

Next, we let $\tilde u$ be fixed as in Lemma~\ref{eq: solappro} and such that $\tilde u(0)=1$ (this normalization does not matter in the following). Moreover, for $\phi \in \cU$, we define
\begin{equation}
  \label{eq:def-u-oben-phi}
u^\phi  \in  C^{2,\alpha}(\overline \O),\qquad u^\phi(\sigma,t)= \tilde u(\phi(\sigma)t).
\end{equation}
We then have
\begin{equation}
\label{eq:Proe1-1-0}
    -\Delta_{g_\phi} u^\phi= 1 \quad \textrm{in}  \quad \O
\end{equation}
since we can write $u^\phi = \bar u\, \circ \Psi_\phi$ with the function
\begin{equation}
  \label{eq:def-bar-u}
\bar u: S^{N-1}\times (-\frac{\pi}{2},\frac{\pi}{2}) \to \R,\; \bar u(\sigma,\theta)=  \tilde u(\theta)
\end{equation}
which solves
\begin{equation}
  \label{eq:approx-sol-eq}
-\Delta_{g_{(\sigma,\theta)}} \bar u = 1\qquad \text{in $S^{N-1} \times (-\frac{\pi}{2},\frac{\pi}{2})$.}
\end{equation}

We need the following fact.
\begin{Lemma}
 \label{sec:peri-solut-serr}
The map
$$
h: \cU  \to C^{1,\alpha} (S^{N-1}), \qquad h(\phi) := \partial_{\nu_\phi} u^\phi(\cdot,1) \in C^{1,\alpha} (S^{N-1})
$$
is smooth, and its derivative at a constant function $\phi \equiv  \lambda>0$ satisfies
\begin{equation}\label{eq:deriv-form}
  \bigl[D_\phi |_{\phi= \lambda}\, h\bigr]\omega= \tilde u''(\lambda)  \omega  \qquad \text{for $\sigma \in S^{N-1}$ and $\omega \in C^{2,\alpha}(S^{N-1})$.}
\end{equation}
 \end{Lemma}

 \proof
For $\phi \in \cU$, recalling the definition of the outer normal $\mu_\phi$ on $\Omega_\phi$ given in (\ref{eq:eta-normal}), we define
$$
\pmb{\mu}_\phi \in  C^{1,\alpha} (S^{N-1}), \qquad {\pmb{\mu}}_\phi(\sigma)= \mu_\phi(\Psi_\phi (\sigma,1))=\mu_\phi(\sigma,\phi(\sigma)).
$$
By (\ref{eq:rel-eta-phi-nu-phi}) and since $\Psi_{\phi}:(\overline \O,g_\phi) \to
(\overline \O_{\phi},g_{(\sigma,\theta)})$ is an isometry, we have
\begin{align}\label{eq: cal h}
h(\phi)(\sigma)= \partial_{\nu_\phi} u^\phi(\sigma,1)&=\langle\nu_\phi(\sigma,1), \nabla_{g_{\phi}}u^\phi(\sigma,1)\rangle_{g_{\phi}}\nonumber\\
&=\langle {\pmb{\mu}}_\phi(\sigma), \nabla_{g_{(\sigma,\theta)}}
\bar{u}(\sigma,\phi(\sigma))\rangle_{g_{(\sigma,\theta)}}\quad\text{for
$\sigma\in S^{N-1}$.}
\end{align}
Recalling that $g_{(\sigma,\theta)}=\textrm{d}\theta^2+\cos^2 \theta \,
g_{S^{N-1}}(\sigma)$, we get
$$
\nabla_{g_{(\sigma,\theta)}} \bar u(\sigma,\theta)=(0,\tilde u'(\theta)) \in T_\sigma S^{N-1} \times \R \qquad \text{for $(\sigma,\theta) \in S^{N-1} \times (-\frac{\pi}{2},\frac{\pi}{2}),$}
$$
and therefore (\ref{eq: cal h}) yields
$$
 h(\phi)(\sigma)= \langle {\pmb{\mu}}_\phi(\sigma),   (0,\tilde u'(\phi(\sigma)) \rangle_{g_{(\sigma,\theta)}}
\qquad\text{for $\sigma\in S^{N-1}$.}
$$
Consequently, for $\omega \in C^{2,\alpha}(S^{N-1})$ we have
\begin{align}
 [D_\phi |_{\phi= \lambda} \,
h(\phi)]\omega &= \langle  [D_\phi |_{\phi= \lambda} \,
{\pmb{\mu}}_\phi ]\omega ,(0,\tilde u'(\lambda) \rangle_{g_{(\sigma,\theta)}} + \langle  {\pmb{\mu}}_\lambda  ,  (0,\tilde u''(\lambda)\omega \rangle_{g_{(\sigma,\theta)}} \nonumber\\
&= \tilde u'(\lambda) \langle  [D_\phi |_{\phi= \lambda} \,
{\pmb{\mu}}_\phi ]\omega , {\pmb{\mu}}_\lambda   \rangle_{g_{(\sigma,\theta)}} + \tilde u''(\lambda)\omega
 \qquad \text{on $S^{N-1}$.} \label{I-1-I-2}
\end{align}
 Here we used in the last step that
$$
{\pmb{\mu}}_\lambda(\sigma) = (0,1) \in T_{\sigma}S^{N-1} \times \R \qquad
\text{for all $\sigma \in S^{N-1}$}
$$
by Lemma~\ref{eq: phi=lambda-lemma}. Moreover, the first term in (\ref{I-1-I-2}) vanishes since the map
$$
\cU
\to C^{1,\alpha}(S^{N-1}), \qquad \phi \mapsto
|{\pmb{\mu}}_\phi|^2_{g_{(\sigma,\theta)}} \equiv 1
$$
is constant and thus
\begin{equation*}
0 = [D_\phi\, |{\pmb{\mu}}_\phi|^2_{g_{(\sigma,\theta)}}] \omega= 2 \langle [D_\phi
\, {\pmb{\mu}}_\phi]\omega
,{\pmb{\mu}}_\phi\rangle_{g_{(\sigma,\theta)}}\quad\textrm {on
$S^{N-1}$}
\end{equation*}
 for every $\phi \in \cU, \: \omega \in C^{2,\alpha}(S^{N-1})$. Thus (\ref{I-1-I-2}) yields
$$
  \bigl[D_\phi |_{\phi= \lambda}\, h\bigr]\omega = \tilde u''(\lambda)  \omega  \qquad \text{for $\omega \in C^{2,\alpha}(S^{N-1})$,}
$$
as claimed.
 \QED

We are now in a position to complete the\\

{\bf Proof of Proposition~\ref{sec:peri-solut-serr-2}.} For $\phi \in
\cU$, we note that the function
\begin{equation}\label{eq: afunc}
a_\phi:= u_\phi-u^\phi\in
C^{2,\alpha}(\overline \Omega)
\end{equation}
 satisfies
\begin{equation}
\label{eq:Proe1-2}
 \left \{
 \begin{aligned}
    -\Delta_{g_\phi} a_\phi&=0 &&\qquad \textrm{ in $\O$}\\
a_\phi&= -u^\phi \qquad &&\qquad \textrm{on $\partial \O$}.
  \end{aligned}
\right.
\end{equation}
Moreover,  in the case where $\phi \equiv \lambda>0$, we have
$$
u^\lambda(\sigma,t)=  \tilde u (\lambda t)\qquad \text{and}\qquad
u_\lambda(\sigma,t)= \tilde u (\lambda t)-\tilde u(\lambda)
$$
by (\ref{eq:form-u-lambda}), and hence
\begin{equation}\label{eq: wcons}
a_\lambda \equiv  -\tilde u(\lambda) \equiv const \qquad \text{on $\overline \Omega$.}
 \end{equation}
We now fix $\omega \in C^{2,\alpha}(\overline \Omega)$, and we consider the smooth map $T: \cU \to C^{0,\alpha}(\overline \Omega)$ given by
$$
 T(\phi) = \Delta_{g_\phi} a_\phi = g_{\phi}^{ij} \partial_{ij} a_{\phi} + \frac{1}{\sqrt{|g_\phi|}} \partial_i \Bigl(\sqrt{|g_\phi|}g_{\phi}^{ij}\Bigr)\partial_j a_\phi.
$$
By (\ref{eq:Proe1-2}) we have $T \equiv 0$ on $\cU$, and thus
\begin{equation}
  \label{eq:0-identity}
0= D T(\phi) \omega =  \Delta_{g_\phi} [D_\phi a_\phi]\omega   + h_{\phi,\omega}^{ij} \partial_{ij} a_{\phi} + \ell^j_\phi \partial_j a_\phi \qquad \text{for $\phi \in \cU$}
\end{equation}
with
$$
h_{\phi,\omega}^{ij}:= [D_\phi  g_{\phi}^{ij}]\omega \qquad \text{and}\qquad \ell^j_\phi := \bigl[D_\phi \frac{1}{\sqrt{|g_\phi|}} \partial_i \bigl(\sqrt{|g_\phi|}g_{\phi}^{ij}\bigr)\bigr]\omega.
$$
Evaluating (\ref{eq:0-identity}) at $\phi= \lambda$ and using \eqref{eq:
wcons}, we find that
\begin{equation}\label{eq: laplas}
 \Delta_{g_\lambda} \tau_{\omega,\lambda} =0 \qquad \text{in $C^{0,\alpha}(\overline \Omega)$}\qquad \text{for $\tau_{\omega,\lambda}:= [D_\phi \big|_{\phi=\lambda} a_\phi] \omega \in C^{2,\alpha}(\Omega)$.}
\end{equation}
Moreover,  the boundary condition in (\ref{eq:Proe1-2})  yields
$$
a_{\phi}(\sigma,1)=-u^{\phi}(\sigma, 1)=-\tilde{u}(\phi(\sigma))
$$
and by differentiation at $\phi \equiv \lambda$ we get
\begin{equation}\label{eq: bord}
\tau_{\omega,\lambda}(\sigma,1) = - \tilde u'(\lambda)
\omega(\sigma)\quad\textrm{for}\quad  \sigma \in S^{N-1}.
\end{equation}
By Lemma~\ref{sec:peri-solut-serr-3}(ii) and \eqref{eq: wcons} we
also have that
\begin{equation}
  \label{eq:proof-linearization-1}
[D_\phi \big|_{\phi=\lambda}   \partial_{\nu_{\phi}} a_\phi
(\cdot,1)]\omega  = \underbrace{\partial_{\tilde \nu_{\phi}(\omega)}
a_\lambda (\cdot,1)}_{=0} + \partial_{\nu_\lambda}
\tau_{\omega,\lambda} (\cdot,1) = \frac{1}{\lambda} \partial_t
\tau_{\omega,\lambda}(\cdot,1)\quad \textrm{on}\quad S^{N-1}.
\end{equation}
 Using Lemma~\ref{sec:peri-solut-serr} and (\ref{eq:proof-linearization-1}), we thus find that
 \begin{align*}
 [D_\phi \big|_{\phi=\lambda}    H(\phi)]\omega=  [D_\phi \big|_{\phi=\lambda}   \partial_{\nu_{\phi}} u_\phi (\cdot,1)]\omega&=  [D_\phi \big|_{\phi=\lambda}   \partial_{\nu_{\phi}} ( a_\phi+u^\phi )(\cdot,1)]\omega\\
&= \frac{1}{\lambda}
\partial_{t} \tau_{\omega,\lambda}(\cdot,1)+ \tilde u''(\lambda) \omega \qquad \text{on $S^{N-1}$.}
 \end{align*}
Since $\tilde u'(\lambda)<0$ by Lemma~\ref{eq: solappro}, we may consider the function $\psi_{\omega,\lambda}:= -\frac{\tau_{\omega,\lambda}}{\tilde
u'(\lambda)}  \in C^{2,\alpha}(\overline \Omega)$ which solves (\ref{eq:psi-omega}) as a consequence of \eqref{eq:
laplas} and \eqref{eq: bord}. We then have
$$
 [D_\phi \big|_{\phi=\lambda}    H(\phi)]\omega = -\frac{\tilde u'(\lambda)}{\lambda}
\partial_{t} \psi_{\omega,\lambda}(\cdot,1)+ \tilde u''(\lambda) \omega \qquad \text{on $S^{N-1}$,}
$$
as claimed in   \eqref{eq:diff-H-express}.
\QED

\begin{Remark}
\label{evenness}
We note that, for given $\omega \in C^{2,\alpha}(S^{N-1})$, the unique solution $\psi_{\omega,\lambda}$ of (\ref{eq:psi-omega}) is even in the $t$-variable. This merely follows from the unique solvability of (\ref{eq:psi-omega}) and the fact that the coefficients of the operator $\Delta_{g_\lambda}$ are even in $t$.
\end{Remark}

\section{Spectral properties of the linearization}
\label{sec:spectr-prop-line-1}

In this section, we provide a detailed study of the spectral properties of the linearized operator $\mbL_\lambda$ given
in Proposition~\ref{sec:peri-solut-serr-2}. For this we first recall that, if
$Y \in C^{\infty}(S^{N-1})$ is a
spherical harmonic of degree $j \in \N \cup \{0\}$, then we have
\begin{equation}
  \label{eq:def-gamma-j}
-\Delta_{S^{N-1}} Y = \gamma_{j} Y \qquad \text{with}\quad
\gamma_j:=  j(N-2+j).
\end{equation}
This property allows to determine the eigenvalues of $\mbL_\lambda$,
as we shall see in the following proposition.

\begin{Proposition}\label{proposition-eigenvalues}
Let $\lambda \in (0,\frac{\pi}{2})$, and let $Y \in C^{\infty}(S^{N-1})$ be a
spherical harmonic of degree $j \in \N \cup \{0\}$. Then $Y$ is an eigenfunction of $\mbL_{\lambda}$. More precisely,  we have
\begin{equation}
  \label{eq:eigenv-eq}
\mbL_\lambda Y = \sigma_{j}(\lambda)\, Y \qquad \text{on
$S^{N-1}$,}
\end{equation}
with
\begin{equation}\label{vapro}
\sigma_{j}(\lambda):= \tilde u'(\lambda) \Bigl((N-1) \tan \lambda - f_j(\lambda)\Bigr)-1,
\end{equation}
where  $\tilde u$ is the solution to \eqref{eq: lapalsolini} and $f_j:[0,\frac{\pi}{2}) \to \R$ is the unique solution of the initial value problem
\begin{equation}
\label{equation-f-j}
\left\{
  \begin{aligned}
f_j'(\l)&=-f_j^2(\l)+(N-1)\tan \l f_j(\l)+ \frac{\g_j}{\cos^2 \l}, \qquad \text{$\l\in (0,\frac{\pi}{2})$}\\
 f_j(0)&=0.
\end{aligned}
\right.
\end{equation}
Moreover, we have
\begin{equation}
\label{equation-f-j-explicit-1}
f_0(\l)=0,\quad f_1(\l)= (N-1) \tan \l \qquad \text{for $\lambda \in [0,\frac{\pi}{2})$}
\end{equation}
and
\begin{equation}
\label{equation-f-j-explicit}
c_j \tan \l \le f_j(\l) \le N  \frac{\sqrt{\gamma_j}}{ \cos \lambda}  \qquad \text{for $\lambda \in
[0,\frac{\pi}{2})$ if $j \ge 2$,}
\end{equation}
where $c_j:= \frac{N-2+\sqrt{(N-2)^2+4\gamma_j}}{2} >N-1$.
\end{Proposition}

\proof In view of Remark~\ref{evenness} and (\ref{eq:def-gamma-j}), it is natural to make a separation of variables ansatz for the unique solution
$\psi_{Y,\lambda}$ of (\ref{eq:psi-omega}) in the special case $\omega = Y$, writing
$\psi_{Y, \l}(\sigma,t)= b_{j,\lambda}(|t|) Y(\sigma)$, for some function  $b_{j,\lambda}:[0,1) \to \R$ with $b_{j,\lambda}'(0)=0$.
By direct computation using (\ref{eq:laplace-lambda}) and (\ref{eq:def-gamma-j}), the problem (\ref{eq:psi-omega}) then reduced to the following boundary value problem for the function $b_{j,\lambda}$:
\begin{equation}\label{eq:problem4}
\left\{
\begin{aligned}
&b_{j,\lambda}''(t) - (N-1) \lambda
\tan(\lambda t) b_{j,\lambda}'(t) - \dfrac{\lambda^2 \gamma_j}{\cos^2(\lambda
t)}b_{j,\lambda}(t) =
 0, \qquad t \in [0,1],\\
&b_{j,\lambda}'(0)=0,\: b_{j,\lambda}(1)=1,
\end{aligned}
\right.
\end{equation}
To see that \eqref{eq:problem4} is uniquely solvable, we make the change of variable $z= \sin(\lambda t)$ which also extracts the $\lambda$-dependence of the problem. We thus write
\begin{equation}
  \label{eq:def-b-B}
b_{j,\lambda}(t)=\frac{B_j(\sin(\l t))}{B_j(\sin \l)},
\end{equation}
where the function $B_j$ is the unique solution of the linear initial value problem
\begin{equation}\label{eq: Bj-initial}
\left\{
  \begin{aligned}
&B''_j -\dfrac{N z}{1-z^2} B'_j-\dfrac{
\gamma_j}{(1-z^2)^2}B_j=0, &&\qquad \text{$z \in (-1,1)$}\\
&B_j(0)=1, \qquad B_j'(0)=0,
  \end{aligned}
\right.
\end{equation}
which is well defined on $[0,1)$. We note that the differential equation in \eqref{eq: Bj-initial}  is a particular case of Heun's equation, see e.g. \cite[Page 57]{Erdil-V3}. It is important to note that
\begin{equation}
  \label{eq:Bj-pos-incr}
\text{$B_j$ is positive and satisfies $B_j'>0$ on $(0,1)$.}
\end{equation}
Indeed, by the differential equation in (\ref{eq: Bj-initial}), any point
$z \in (-1,1)$ with $B_j'(z)=0$ is a strict local minimum if
$B_j(z)>0$ and a strict local maximum if $B_j(z)<0$. Combining this property with
the initial conditions in (\ref{eq: Bj-initial}), we obtain (\ref{eq:Bj-pos-incr}).

In particular, since $\lambda \in (0,\frac{\pi}{2})$, it follows from (\ref{eq:Bj-pos-incr}) that $B_j(\sin \l)>0$. Therefore,
the function $b_{j,\lambda}$ is well defined on $[0,1]$ by (\ref{eq:def-b-B}), and a direct computation shows that it solves (\ref{eq:problem4}).

By definition of the operator $\mbL_\lambda$ and the special form of the function $\psi_{Y,\lambda}$, we then conclude that (\ref{eq:eigenv-eq}) holds with
$$
\sigma_j(\lambda)= -\frac{\tilde
u'(\lambda)}{\lambda}b_{j,\lambda}'(1)+\tilde u''(\lambda) = \tilde
u'(\lambda) \Bigl((N-1) \tan \lambda - \frac{b_{j,\lambda}'(1)}{\lambda}\Bigr)-1,
$$
where we used in the last step that
\begin{equation}
\label{eq: seconderini}
\tilde{u}''(\l)=-1+(N-1)\tan \l \tilde{u}'(\l) \qquad \text{for $\lambda \in [0,\frac{\pi}{2})$}
\end{equation}
by (\ref{eq: lapalsolini}). Next we prove that (\ref{equation-f-j}) admits a unique solution $f_j$ on $[0,\frac{\pi}{2})$, and that
\be \label{eq:def-f-lambda}
f_j(\l)=\frac{b_{j,\lambda}'(1)}{\l} \qquad \text{for $\lambda \in (0,\frac{\pi}{2})$.}
\ee
The uniqueness is obvious, since (\ref{equation-f-j}) is an initial value problem for a first order ODE with smooth coefficients. We now observe that (\ref{equation-f-j}) is solved by the function
\be
\label{eq:f-cos-Bj-pr}
f_j:[0,
\frac{\pi}{2}) \to \R, \qquad f_j(\l):=\cos \l \frac{B_j'(\sin \l
)}{B_j(\sin \l)}. \ee
Indeed, differentiation yields
\begin{align}
f_j'(\lambda) &= -\sin \l\: \frac{B_j'(\sin \l )}{B_j(\sin \l)} + \cos^2 \l\: \frac{B_j''(\sin \l)B_j(\sin \l)-[B_j'(\sin \l)]^2}{B_j^2(\sin \l)} \nonumber\\
&= -\tan \l\: f_j(\l)  + \cos^2 \l\: \frac{B_j''(\sin \l)}{B_j(\sin
\l)} -f_j^2(\l),  \label{f_j-eq-calc}
\end{align}
and
$$
\cos^2 \l\: B_j''(\sin \l)=N {\sin \l}{} \,B_j'(\sin \l
)+\frac{\g_j B_j(\sin \l)}{\cos^2 \l}
$$
by (\ref{eq: Bj-initial}). Inserting this in (\ref{f_j-eq-calc}) gives (\ref{equation-f-j}), as claimed. Moreover, by (\ref{eq:def-b-B}) we have
\begin{equation}\label{eq:calvj-0}
b'_{j,\lambda}(1)=\frac{\l \cos \l \: B'_j(\sin \l)}{B_j(\sin \l)}=
\l f_j(\l) \qquad \text{for $\lambda \in (0,\frac{\pi}{2})$,}
\end{equation}
as claimed in \eqref{eq:def-f-lambda}.

Next we note that (\ref{equation-f-j-explicit-1}) is an obvious consequence of (\ref{equation-f-j}). To complete the proof, it thus remains to prove (\ref{equation-f-j-explicit}). So fix $j \ge 2$. Then $\gamma_j >N-1$, and thus the unique positive root $c=c_j$ of the equation
$c^2- (N-2)c-\gamma_j=0$ satisfies $c_j= \frac{N-2 + \sqrt{(N-2)^2+4\gamma_j}}{2} >N-1$. Moreover, the function
$$
\lambda \mapsto g_j(\lambda)= c_j \tan \l
$$
satisfies $g(0)=0$ and
\begin{align*}
g_j'(\lambda)+ g_j^2(\lambda)- (N-1) \tan \l\: g_j(\l)- \frac{\gamma_j}{\cos^2 \l}&= \frac{c_j+ \sin^2 \l\:[c_j^2-(N-1)c_j]- \gamma_j}{\cos^2 \l}\\
&\le \frac{c_j^2-(N-2)c_j- \gamma_j}{\cos^2 \l}=0
\end{align*}
for $\lambda \in (0,\frac{\pi}{2})$. Hence the function $h_j= f_j-g_j$ satisfies $h_j(0)=0$ and
$$
h_j'(\lambda) \ge g_j^2(\lambda) - f_j^2(\lambda) + (N-1) \tan \l\: h_j(\l) = \Bigl((N-1) \tan \l - f_j(\lambda)-g_j(\lambda)\Bigr) h_j(\lambda).
$$
By Gronwall's inequality, this implies that $h_j \ge 0$ on
$[0,\frac{\pi}{2})$,  and thus
$$
f_j(\lambda) \ge g_j(\lambda) = c_j \tan \l \qquad \text{for $\lambda \in [0,\frac{\pi}{2})$,}
$$
which gives the lower estimate in (\ref{equation-f-j-explicit}).

To see the upper bound in (\ref{equation-f-j-explicit}), we argue by contradiction. So we consider the function
$\lambda \mapsto F_j(\lambda):= f_j(\l) - N  \frac{\sqrt{\gamma_j}}{ \cos \lambda}$ and we assume that the set
$$
M:= \{\lambda \in (0,\frac{\pi}{2})\::\: F_j(\lambda) \ge 0\}
$$
is nonempty. Since $F_j(0)=-N \sqrt{\gamma_j}<0$ and $F_j$ is continuous, it then follows that there exists
$\lambda^*:= \min M \in (0,\frac{\pi}{2})$. We then have $F_j(\lambda^*)=0$, i.e., $f_j(\lambda^*)= N  \frac{\sqrt{\gamma_j}}{ \cos \lambda^*}$ and
\begin{align*}
0 &\le F_j'(\lambda^*)= -f_j^2(\lambda^*) +(N-1)\tan \l^* \: f_j(\lambda^*)+ \frac{\g_j}{\cos^2 \l^*} -
N \sqrt{\gamma_j} \frac{\sin \l^*}{\cos^2 \l^*}\\
& \le -N^2  \frac{\gamma_j}{ \cos^2 \l^*} +  N(N-1) \tan \l^* \frac{\sqrt{\gamma_j}}{\cos \l^*} + \frac{\g_j}{\cos^2 \l^*}\\
&\le \frac{1}{\cos^2 \l^*} \Bigl((1-N^2)\gamma_j  + N(N-1) \sqrt{\gamma_j}\Bigr) \le \frac{N^2-1}{\cos^2 \l^*} \bigl(\sqrt{\gamma_j}- \gamma_j\bigr)<0.
\end{align*}
This is a contradiction, and thus the upper bound in (\ref{equation-f-j-explicit}) follows. The proof is finished.
 \QED

Proposition~\ref{proposition-eigenvalues} allows us to easily derive important properties of the shape and asymptotics of the eigenvalue curves $\lambda \mapsto \sigma_j(\lambda)$.
The following lemmas will be crucial
in order to apply the Crandall-Rabinowitz theorem.

\begin{Lemma}
\label{j-asymptotics}
For every $\lambda \in (0,\frac{\pi}{2})$ we have
$$
0 < \liminf_{j \to \infty}\frac{\sigma_j(\lambda)}{j} \le \limsup_{j \to \infty}\frac{\sigma_j(\lambda)}{j}<\infty.
$$
\end{Lemma}

\begin{proof}
Since $\tilde u'(\lambda)<0$ for $\lambda \in (0,\frac{\pi}{2})$ by (\ref{ine der}), the claim follows from (\ref{vapro}) and (\ref{equation-f-j-explicit}).
\end{proof}

\begin{Lemma}\label{eq:Uniquezero}
For every $\lambda \in (0,\frac{\pi}{2})$ and $i,j \in \N \cup \{0\}$, $i < j$ we have
$\sigma_j(\lambda)>\sigma_i(\lambda)$.
\end{Lemma}

\proof This follows readily from (\ref{vapro}) and the fact that
\begin{equation}
  \label{eq:order-f-j}
f_j(\lambda)>f_i(\lambda) \qquad \text{for $i,j \in \N \cup \{0\}$, $i < j$.}
\end{equation}

To see (\ref{eq:order-f-j}), note that $h:= f_j-f_i$ solves the initial value problem
\begin{equation}
\label{equation-h-proof}
\left\{
  \begin{aligned}
h'(\lambda) &= a(\lambda)  h(\lambda)+ \frac{\gamma_j-\gamma_i}{\cos^2 \l},\qquad \text{$\lambda \in (0,\frac{\pi}{2})$,}\\
 h(0)&=0,
\end{aligned}
\right.
\end{equation}
with $a(\lambda)=(N-1) \tan \l - f_j(\lambda)-f_i(\lambda)$. Since
$\gamma_j-\gamma_i>0$, this implies that $h(\lambda)>0$ for $\lambda
\in (0,\frac{\pi}{2})$, and thus (\ref{eq:order-f-j}) holds.
 \QED

\begin{Lemma}\label{eq: eges}
The eigenvalue curves $\lambda \mapsto \sigma_j(\lambda)$ have the following asymptotics.
\begin{enumerate}
\item $\lim \limits_{\l \to 0} \sigma_j(\l)=-1$ for every $j \ge 0$.
\item $\lim \limits_{\l \to \frac{\pi}{2}^-} \sigma_j(\l)=\infty $ for every $j \ge 2$.
\end{enumerate}
\end{Lemma}

\proof
 Claim $(i)$ follows immediately from (\ref{vapro}) and (\ref{equation-f-j}).\\
Moreover, for $j \ge 2$, (\ref{ine der}), (\ref{vapro}) and (\ref{equation-f-j-explicit}) yield that
$$
\sigma_j(\l) \ge  (c_j-N-1) \l \tan \l \to \infty \qquad \text{as
$\l \to \frac{\pi}{2}^-$,}
$$
as claimed.
\QED

\begin{Proposition}\label{eq:OK-egesff}
For every $j\geq 2$,  there exists a unique $\l_j\in (0, \frac{\pi}{2})$ such that $\sigma_j(\lambda_j)=0$. Moreover, we have
\begin{enumerate}
\item $\sigma_j'(\lambda_j) >  0$ for $j \ge 2$;
\item $\lambda_j<\lambda_i$ for $2 \le i < j$;
\item $\lim \limits_{j \to \infty} \lambda_j=0$.
\end{enumerate}
 \end{Proposition}

\proof
Let  $j\geq2$, and suppose that $\l_* \in (0,\frac{\pi}{2})$ satisfies $\sigma_j(\l_*)=0$.  We first claim that
\begin{equation}
  \label{eq:claim-sigma-prime-pos}
\sigma'_j(\l_*)>0.
\end{equation}
Indeed, differentiating (\ref{vapro}) at $\lambda= \lambda_*$  gives
\begin{equation}
  \label{eq:claim-sigma-prime-pos-1}
\s_j'( \l_*)=  \tilde u''(\lambda_*) \Bigl((N-1)\tan \tilde \l -
f_j(\l_* )\Bigr) +  \tilde u'(\lambda_*)
\Bigl(\frac{N-1}{\cos^2 \l_*} -  f_j'(\l_* )\Bigr).
\end{equation}
Moreover, (\ref{vapro}) gives
$$
0= \sigma_j(\l_*)= (N-1) \tan \lambda_*\: \tilde u'(\lambda_*)   - f_j(\lambda_*)\tilde u'(\lambda_*)-1,
$$
and thus, by \eqref{eq: seconderini},
$$
\tilde{u}''(\l_*)=-1+(N-1)\tan \l_*\: \tilde{u}'(\l_*)= f_j(\lambda_*)\tilde{u}'(\l_*).
$$
Inserting this in (\ref{eq:claim-sigma-prime-pos-1}) and using (\ref{equation-f-j}), we deduce that
$$
\s_j'(\l_*) =\tilde{u}'(\l_*)\Bigl((N-1) \tan \l_* \:f_j(\l_* )- f_j^2(\l_* )+\frac{N-1}{\cos^2 \l_*}- f_j'(\l_* )\Bigr)= \tilde{u}'(\l_*)\frac{N-1-\g_j}{\cos^2 \l_*}.
$$
Since  $\tilde{u}'<0$ on $(0,\frac{\pi}{2})$ and $\g_j>N-1$ for $j\geq 2$, (\ref{eq:claim-sigma-prime-pos}) follows.

By combining (\ref{eq:claim-sigma-prime-pos}), Lemma~\ref{eq: eges} and the intermediate value theorem, it then follows that
for every $j \ge 2$ there exists a unique $\l_j\in (0, \frac{\pi}{2})$ such that $\sigma_j(\lambda_j)=0$, and that (i) holds. Claim (ii) then follows immediately from
Lemma~\ref{eq:Uniquezero}. Moreover, by (\ref{ine der}), (\ref{vapro}) and (\ref{equation-f-j-explicit}) we have
$$
1 = \sigma(\lambda_j)+1 = \tilde u'(\lambda_j) \Bigl((N-1) \tan \lambda_j     - f_j(\lambda_j) \Bigr) \ge (c_j -N+1) \lambda_j  \tan \lambda_j
$$
for $j \ge 2$. Since $c_j \to \infty$,  it follows that $\lambda_j \to 0$ as $j \to \infty$, as claimed in (iii).
  \QED

\section{Mapping properties in Sobolev spaces}
\label{sec:function-spaces}

In the following, we consider the Sobolev spaces
\begin{equation}
  \label{eq:def-hpe}
  H^{k} := H^k(S^{N-1}), \qquad k \in \N \cup
\{0\},
\end{equation}
and as usual we put $L^2:= H^{0}$. Note that $L^2$ is a Hilbert space with scalar product
$$
(u,v) \mapsto \langle u,v \rangle_{L^2} := \int_{S^{N-1}} u v\,d\sigma \qquad \text{for $u,v \in L^2$.}
$$
In the following, for $\ell \in \N \cup \{0\}$, we let
\begin{equation}
  \label{eq:def-Vell}
V_\ell \,\subset \,\bigcap_{k \in \N} H^k
\end{equation}
denote the space of spherical harmonics of degree $\ell$. Moreover, we let $P_\ell: L^2 \to L^2$ denote the $\langle \cdot,\cdot \rangle_{L^2}$-orthogonal projections on $V_\ell$ and
\begin{equation}
  \label{eq:def-Wjell}
W^k_\ell := \{v \in H^k\::\: P_\ell v = 0\} \:\subset \: H^k,
\end{equation}
the $\langle \cdot,\cdot \rangle_{L^2}$-orthogonal complements of $V_\ell$ in $H^k$.
We note that the space $H^k$ is characterized by the subspace of functions $v\in L^2$ such that
 \begin{equation}
\label{eq:scp-hj}
\sum_{\ell \in (\N \cup \{0\})} (1+\ell^2)^{k}
\langle P_\ell v, P_\ell v \rangle_{L^2} < \infty
\end{equation}
We may then define a scalar product on $H^k$ by setting
\begin{equation}
  \label{eq:scp-hjA}
(u,v) \mapsto \langle u,v \rangle_{H^k} :=  \sum_{\ell \in (\N \cup \{0\})} (1+\ell^2)^{k}
\langle P_\ell u, P_\ell v \rangle_{L^2}  \qquad \text{for $u,v \in H^k$,}
\end{equation}
and the induced norm of this scalar product is equivalent to the standard norm on $H^k$.
\begin{Proposition}
\label{sec:spectr-prop-line} For fixed $\lambda \in
(0,\frac{\pi}{2})$, the linear map $\mbL_{\lambda}$ defined in
\eqref{eq:diff-H-express} extends to a continuous linear map
\begin{equation}
  \label{eq:prop-extension-1-0}
\mbL_\lambda: H^{2} \to H^{1},\qquad \mbL_\lambda v =
\sum_{\ell \in \N \cup \{0\}} \sigma_\ell(\lambda) P_\ell v.
\end{equation}
Moreover, for any $\ell \in \N \cup \{0\}$, the operator
\begin{equation}
  \label{eq:isomorphism-property}
\mbL_\lambda - \sigma_\ell(\lambda) \id \;:\;  W^2_\ell \to W^1_\ell \qquad \text{is an isomorphism.}
\end{equation}
\end{Proposition}
\proof
From the linear asymptotic upper growth estimate for the values $\sigma_j(\lambda)$ given in Lemma~\ref{j-asymptotics}, it readily follows that (\ref{eq:prop-extension-1-0}) defines a continuous linear operator $\mbL_\lambda: H^{2} \to H^{1}$. On finite linear combinations of spherical harmonics, this operator coincides with the operator $\mbL_\lambda: C^{2,\alpha}(S^{N-1}) \to  C^{1,\alpha}(S^{N-1})$ given in \eqref{eq:diff-H-express} by construction. Since linear combinations of spherical harmonics are dense both in $C^{2,\alpha}(S^{N-1})$ and in $H^2$, the claimed extension property follows by continuity.

Finally, to show (\ref{eq:isomorphism-property}), we fix $\ell \in \N \cup \{0\}$, and we first note that $\mbL_\lambda$ indeed maps $W^2_\ell$ into $W^1_\ell$ by definition. Moreover, the lower growth estimate for the values $\sigma_j(\lambda)$ given in Lemma~\ref{j-asymptotics} shows that the linear operator
$$
W^1_\ell \to W^2_\ell, \qquad w \mapsto \sum_{\stackrel{m \in \N \cup \{0\}}{m \not=\ell}} \frac{1}{\sigma_m(\lambda)-\sigma_\ell(\lambda)}P_m w
$$
is well-defined and continuous. Moreover, by direct computation, its inverse  is  $\mbL_\lambda - \sigma_\ell(\lambda) \id \;:\;  W^2_\ell \to W^1_\ell$. Hence (\ref{eq:isomorphism-property}) is proved.
\QED

\begin{Remark}
  \label{sec:spectr-remark-1}
Since operator $\mbL_\lambda: H^2 \to H^1$ given in (\ref{eq:prop-extension-1-0}) is a continuous linear extension of the operator $\mbL_\lambda:
C^{2,\alpha}(S^{N-1}) \to C^{1,\alpha}(S^{N-1})$ defined in Proposition~\ref{sec:peri-solut-serr-2}, it can also be characterized as follows. Given $\omega \in H^2$, let  $\psi \in W^{1,2}(\Omega)$ be the unique weak solution of the problem
\begin{equation}
  \label{eq:psi-omega-0}
\left \{
  \begin{aligned}
 &\Delta_{g_\lambda} \psi =0 && \qquad \text{in  $\O$},\\
&\psi (\sigma,t)= \omega(\sigma)  && \qquad \text{for $(z,t) \in \partial \O$,}
  \end{aligned}
\right.
\end{equation}
which is axially symmetric in the $\sigma$-variable,  where $\Delta_{g_\lambda}$ is given by \eqref{eq:laplace-lambda}. By standard elliptic regularity theory, we then have $\psi \in W^{2,2}(\Omega)$, and $\mbL_\lambda \omega$ is given by
$$
[\mbL_\lambda \omega] (\sigma) = -\frac{\tilde u'(\lambda)}{\lambda}
\partial_{t}\, \psi (\sigma,1) +\tilde u''(\lambda) \omega(\sigma)
\qquad \text{for a.e. $\sigma \in S^{N-1}$,}
$$
where $\partial_{t} \psi$ is considered in the sense of traces. This can be easily seen by approximating $\omega$ in $H^2$ with functions in $C^{2,\alpha}(\Omega)$ and using standard elliptic estimates.
\end{Remark}

\section{Proof of Theorem~\ref{teo1}}
\label{sec:compl-proof-theor}

In the following, we consider the spaces
\begin{align*}
&\calX:= \{\vp \in C^{2,\alpha}(S^{N-1})\::\: \text{$\vp$ axially symmetric}\},\\
&\calY:= \{\vp \in C^{1,\alpha}(S^{N-1})\::\: \text{$\vp$ axially symmetric}\}.
\end{align*}
Here axial symmetry is defined with respect to the axis $\R e_1$ as in the introdution, see (\ref{eq:def-axial-symmetry}).
We also consider the nonlinear operator $H$ defined in (\ref{eq:def-H}), and we note that $H$ maps $\cU \cap \calX$ into $\calY$
by Lemma~\ref{sec:peri-solut-serr-1}(iii). Consider the open set
\begin{equation}
  \label{eq:def-cO}
\cO:= \{(\lambda,\vp) \in (0,\frac{\pi}{2}) \times \calX \::\: - \lambda
<\vp <- \lambda+\frac{\pi}{2}\}  \subset \R \times \calX.
\end{equation}
In view of \eqref{eq:def-H} and \eqref{calcH-lambda}, the proof of Theorem~\ref{teo1} will be completed by applying the Crandall-Rabinowitz bifurcation theorem (see Theorem~\ref{eq: Cradal Rabi}) to the smooth nonlinear operator
\be \label{eq:def--G}
G: \cO \subset \R \times \calX  \to \calY, \qquad G(\lambda, \vp) =H(\lambda+\vp)- H(\l)=
H(\lambda+\vp)-\tilde u'(\lambda).
\ee
By \eqref{calcH-lambda}, we have $G(\lambda,0)=0$ in $\calY$ for all $\lambda \in (0,\frac{\pi}{2})$. Moreover,
\begin{equation}
  \label{eq:cGH-lambda}
 D_\vp G (\lambda,0)= DH(\l)\big|_\calX = \mbL_\lambda|_\calX \in \cL(\calX,\calY).
\end{equation}
 We have the following

\begin{Proposition}\label{propPhi}
Consider  $\lambda_j$, with  $j \ge 2$, as given in Proposition \ref{eq:OK-egesff}.
For $j \ge 2$, the linear operator
$$
\mbL_j:=\mbL_{\lambda_j}\big|_\calX \in \cL(\calX,\calY )
$$
has the following properties.
\begin{itemize}
\item[(i)] The kernel $N(\mbL_j)$ of $\mbL_j$ is spanned by $Y_j$, the
unique $L^2$-normalized real-valued axially symmetric spherical harmonic of degree $j$
defined in the introduction.
\item[(ii)] The range of $\mbL_j$ is given by
$$
R(\mbL_j)= \Bigl \{v \in \calY  \::\: \int_{S^{N-1}} v Y_j\,d\sigma = 0 \Bigr\}.
$$
\end{itemize}
Moreover,
\begin{equation}
  \label{eq:transversality-cond}
\partial_\lambda \Bigl|_{\lambda= \lambda_j} \mbL_\lambda Y_j \;\not  \in \; R(\mbL_j).
 \end{equation}
\end{Proposition}

\proof
By definition we have $\sigma_j(\lambda_j)=0$, which by
Proposition~\ref{proposition-eigenvalues} is equivalent to $\mbL_jY_j  = 0$. We consider the subspaces
\begin{align}
\calX_j &:= \Bigl \{v \in  \calX \::\: \int_{S^{N-1}} v Y_j\,d\sigma = 0  \Bigr\} \subset \calX, \label{def-X*}\\
\calY_j &:= \Bigl \{v \in \calY \::\: \int_{S^{N-1}} v Y_j\,d\sigma = 0 \Bigr\} \subset \calY. \nonumber
\end{align}
To show properties (i) and (ii), it clearly suffices to prove that
\begin{equation}
  \label{eq:isomorphism}
\text{$\mbL_j$ defines an isomorphism between $\calX_j$ and $\calY_j$.}
\end{equation}
To prove (\ref{eq:isomorphism}), we need to introduce further spaces.
We recall the definition of the Sobolev space $H^k$ in (\ref{eq:def-hpe}) and put
\begin{equation}
  \label{eq:def-hpes}
H^{k}_{ax} := \left\{v \in H^k  \::\: \text{$v$ axially symmetric} \right\}, \qquad k \in \N \cup \{0\},
\end{equation}
noting that $\calX = H^{2}_{ax} \cap C^{2,\alpha}(S^{N-1})$ and $\calY= H^{1}_{ax} \cap C^{1,\alpha}(S^{N-1})$. Proposition~\ref{sec:spectr-prop-line} implies that $\mbL_j$ defines a continuous linear operator
\begin{equation}
  \label{eq:formular-H*-perm}
\mbL_j : H^2_{ax} \to H^1_{ax}, \qquad \mbL_j v = \sum_{\ell \in \N \cup \{0\}} \sigma_\ell (\lambda_j) P_\ell v.
\end{equation}
Next we put
$$
\tilde V^{k}_j:=  H^k_{ax} \cap V_j, \quad  \tilde W^{k}_j:=  H^k_{ax} \cap W_j^k \quad \subset H^k_{ax}  \qquad \text{for $k=1,2$,}
$$
where the spaces $V_j$ resp. $W_j^k$ are defined in (\ref{eq:def-Vell}) and (\ref{eq:def-Wjell}), respectively.
  We note that $\tilde V^{k}_j$ is one-dimensional and spanned by the function $Y_j$ given in $(i)$. We then deduce from Proposition~\ref{sec:spectr-prop-line} and the property $\sigma_j(\lambda_j)=0$ that
\begin{equation}
  \label{eq:Sobolev-invertible}
\text{$\mbL_j$ defines an isomorphism $\tilde W^2_j \to \tilde W^1_j$}.
\end{equation}
Moreover, since $\calX_j= \tilde W^2_j \cap \calX$ and $\calY_j = \tilde W^1_j \cap \calY$, we see that $\mbL_j: \calX_j \to \calY_j$
is well defined and injective. To establish surjectivity, we let $f \in \calY_j$. By (\ref{eq:Sobolev-invertible}),
there exists $\omega \in \tilde W^2_j \subset H^2$ such that $\mbL_j \omega=f$.

As noted in Remark~\ref{sec:spectr-remark-1}, we then have
\begin{equation}
   \label{eq:corr-version-add-0}
- \frac{\tilde u'(\lambda_j)}{\lambda_j} \partial_\nu  \psi (\sigma, \pm 1) + \tilde u''(\lambda_j)
  \omega(\sigma)= f(\sigma)\quad \text{for a.e. $\sigma \in S^{N-1}$,}
\end{equation}
where $\psi \in W^{2,2}_{loc}(\Omega)$ is the unique solution of (\ref{eq:psi-omega-0}) with $\lambda= \lambda_j$, recalling that $\O=S^{N-1}\times (-1,1)$.
We claim that
\begin{equation}
  \label{eq:corr-version-add-1}
\psi \in C^{2,\alpha}(\overline \O).
\end{equation}
This regularity property will follow from \cite[Theorem 6.3.2.1]{grisvard} once we have shown that
\begin{equation}
  \label{eq:corr-version-add-2}
\psi \in W^{2,p}(\O) \qquad \text{for all $p \in (1,\infty)$.}
\end{equation}
Indeed, if (\ref{eq:corr-version-add-2}) holds, then Sobolev embeddings also give that $\psi \in C^{1,\alpha}(\overline \O)$ and therefore $\omega \in C^{1,\alpha}(S^{N-1})$ by (\ref{eq:psi-omega-0}). Consequently, \cite[Theorem 6.3.2.1]{grisvard} applies with the order $d=1$ of the boundary operator (see \cite[Section 2.1]{grisvard} and gives (\ref{eq:corr-version-add-1}).

To see (\ref{eq:corr-version-add-2}), we show by induction that
\begin{equation}
  \label{eq:corr-version-add-3}
\psi \in W^{2,p_k}(\O)
\end{equation}
for a sequence of numbers $p_k \in [2,\infty)$ satisfying $p_0=2$ and $p_{k+1} \ge \frac{N-1}{N-2}p_{k}$ for $k \ge 0$. We already know that (\ref{eq:corr-version-add-3}) holds for $p_0=2$. So let us assume that (\ref{eq:corr-version-add-3}) holds for some $p_k \ge 2$. We distinguish two cases.\\
If $p_k < N$, then the trace theorem \cite[Theorem
5.4]{Adams} implies that
$$
\psi \big|_{\partial
\O}\in  W^{1,p_{k+1}}(\partial \O) \qquad \text{with $p_{k+1}:= (\frac{N-1}{N-p_k})p_k \ge \frac{N-1}{N-2} p_k$,}
$$
and thus $\omega \in W^{1,p_{k+1}}(S^{N-1})$. Since also $f \in C^{1,\alpha}(S^{N-1}) \subset W^{1,p_{k+1}}(S^{N-1})$ and, by (\ref{eq:corr-version-add-0}),
$$
- \frac{\tilde u'(\lambda_j)}{\lambda_j} \partial_\nu  \psi (\sigma, \pm 1) + \psi(\sigma, \pm 1)=g(\sigma)\quad \text{for a.e. $\sigma \in S^{N-1}$,}
$$
with
$$
g =   (1-\tilde{u}''(\lambda_j))\omega+f\in W^{1,p_{k+1}}(S^{N-1}),
$$
we may deduce from \cite[Theorem 2.4.2.6]{grisvard} that $\psi \in W^{2,p_{k+1}}(\O)$.\\
If $p_k \ge  N$, the trace theorem implies that $\omega\in W^{1,p}_{loc}(\partial \O)$ for any $p>2$, and then we may repeat the above argument with arbitrarily chosen $p_{k+1} \ge \frac{N-1}{N-2}p_{k}$ to infer again that $\psi \in W^{2,p_{k+1}}(\O)$.\\
We thus conclude that (\ref{eq:corr-version-add-2}) holds, and hence (\ref{eq:corr-version-add-1}) follows. By passing to the trace again, we then conclude that $\omega \in  C^{2,\alpha}(S^{N-1})$. Consequently, $\omega \in C^{2,\alpha}(S^{N-1}) \cap \tilde W^2_j = \calX_j$, and thus $\mbL_j: \calX_j \to \calY_j$ is also surjective.  Hence (\ref{eq:isomorphism}) is true.

We finally derive ~\eqref{eq:transversality-cond} from the identity
\begin{equation}\label{Dlambda}
\partial_\lambda \Bigl|_{\lambda= \lambda_j} \mbL_\lambda Y_j =  \partial_\lambda
\Bigl|_{\lambda= \lambda_j}\sigma_j(\lambda) Y_j = \sigma'_j(\lambda_j) Y_j
\end{equation}
and Proposition \ref{eq:OK-egesff}.
\QED

\noindent {\bf Proof of Theorem~\ref{teo1} (completed).}
Recalling \eqref{eq:def--G} and \eqref{eq:def-H},  we shall  apply the Crandall-Rabinowitz bifurcation Theorem to solve the equation
\be \label{eq:final-eq-to-solve}
G(\l,\vp)=H(\lambda+\vp)-H(\lambda)= \partial_{\nu_{\phi}} u_{\phi} (e_1,\cdot )- \widetilde{u}'(\l)= 0,
\ee
where $\phi=\l+\vp \in \cU$ and  the function  $u_{\phi}  \in C^{2,\a}(\ov{\O_{\phi}})$  is the unique solution to  the Dirichlet boundary value problem
\begin{equation*}
\left\{
  \begin{aligned}
    -\Delta_{g_\phi} u_\phi&= 1 &&\qquad \textrm{in $\O$,}\\
u_\phi&=0 &&\qquad \textrm{on $\partial \O$},
  \end{aligned}
\right.
 \end{equation*}
we obtained from  Lemma \ref{sec:peri-solut-serr-1}. Once this is done,
\eqref{eq:reformulated-overd-H} follows and thus we get \eqref{eq:reformulated-overd}, which is equivalent to \eqref{eq:Proe1-1-neumann} with $-c= H(\lambda)$. \\
 To solve equation \eqref{eq:final-eq-to-solve},   we fix $j \ge 2$. Moreover, we let
  $\lambda_j$ be defined as in Proposition \ref{eq:OK-egesff}, and we let $\calX_j$ be defined as in (\ref{def-X*}).

By Proposition~\ref{propPhi}, the assumptions of the Crandall-Rabinowitz bifurcation theorem -- as stated in Theorem~\ref{eq: Cradal Rabi} in the appendix --
are satisfied with $\calX = \calX_j$, $\calY= \calY_j$, $\cO$ and $G$ defined in (\ref{eq:def-cO}) and \eqref{eq:def--G} and
$$
\cZ = \cZ_j:= \Bigl \{v \in \calX_j  \::\: \int_{S^{N-1}} v Y_j\,d\sigma = 0 \Bigr\}.
$$
Applying this theorem, we then find ${\e_j}>0$ and a smooth curve
$$
(-{\e_j},{\e_j}) \to \cO, \qquad s \mapsto (\lambda_j(s),\varphi_s^j)
$$
such that
\begin{enumerate}
\item[(i)] $G(\lambda_j(s),\varphi_s^j)=0$ for $s \in (-{\e_j},{\e_j})$.
\item[(ii)] $\lambda_j(0)= \lambda_j$.
\item[(iii)]  $\varphi_s^j = s \bigl(Y_j + w_s^j\bigr)$ for $s \in (-\e_j,\e_j)$ with a smooth curve
$$
(-{\e_j},{\e_j}) \to \cZ_j, \qquad s \mapsto w_s^j
$$
satisfying $w_0^j=0$.
\end{enumerate}
Since  $(\lambda_j(s),\varphi_s^j)$ is a solution to
\eqref{eq:final-eq-to-solve} for every $s\in (-\e_j, \e_j)$,  the
function  $u_{\phi^{j}_s} \in C^{2,\a}(\ov{\O})$  solves the
overdetermined boundary value  problem
\begin{equation*}
\left\{
  \begin{aligned}
    -\Delta_{g_{\phi_s^j}} u_{\phi^{j}_s} &=1 &&\qquad \textrm{ in } \quad\Omega ,\\
u_{\phi^{j}_s} &=0,\quad \partial_\nu u_{\phi^{j}_s} =\ti{u}'(\lambda_j(s))
&&\qquad \textrm{ on }\quad \partial\Omega ,
  \end{aligned}
\right.
\end{equation*}
where $\phi^{j}_s=\l_j(s)+\vp_s^j$. It then follows from the
reformulation procedure set up in Section~\ref{sec:transf-probl-its}
 that the map
$s \mapsto (\lambda_j(s), \vp_s^j)$  and the function
$u:=u_{\phi^{j}_s}\circ\Psi_{\phi^j_s}^{-1}\circ\Upsilon^{-1}:
D_{\phi^{j}_s}\to \R$ have the properties asserted in Theorem
\ref{teo1}. \QED

\section{Appendix}
\label{sec:appendix}
We state here a version of the
Crandall-Rabinowitz bifurcation theorem used in Section~\ref{sec:compl-proof-theor}, which is slightly different from but obviously equivalent to \cite[Theorem 1.7]{M.CR}. For the proof and classical applications, see e.g. \cite{M.CR}, \cite{H.Kielhofer} and \cite{J. Smoller}.

\begin{Theorem}[Crandall-Rabinowitz bifurcation theorem] \label{eq: Cradal Rabi}
Let $\calX$  and $\calY$ be two Banach spaces, $\cO \subset \R \times \calX$ be an open set, where the elements of $\R \times \calX$ are denoted by $(\lambda, \varphi)$. Let $I \subset \R$ be an open interval such that $I \times \{0\} \subset\calX$, and let  $G: \cO\rightarrow \calY$  be a twice
continuously differentiable function such that
 \begin{enumerate}
\item
$ G(\lambda, 0)=0\quad \textrm{ for all }\quad \lambda \in I,$
 \item
 $\ker(D_\varphi G(\lambda_*, 0))=\R x_*$  for some $\lambda_*\in I$  and $x_* \in \calX\setminus \{0\}$.
 \item
 $\textrm{Codim Im}(D_\varphi G(\lambda_*, 0))=1$
 \item
 $D_\lambda D_\varphi G(\lambda_*, 0)(x_*) \notin \textrm{Im}(D_\varphi G(\lambda_*, 0))$.
 \end{enumerate}
 Then for any complement $\cZ$ of the subspace $\R x_* \subset \calX$,  there exists a continuous curve
   $$(-\e, \e)\longrightarrow [\R \times \cZ] \cap \cO, \quad s\mapsto (\lambda(s), w(s))$$  such that
 \begin{enumerate}
 \item
 $\lambda(0)=\lambda_*, \quad \varphi(0)=0$
 \item
 $s(x_*+w(s))\in \cO$
 \item
 $G(\lambda(s), s(x_*+w(s))=0$
 \end{enumerate}
 Moreover, the set of solutions to the equation  $G(\lambda, u)=0$ in a neighborhood of  $(\lambda_*, 0)$
 is given by  the curve $\{ (\lambda, 0), \lambda\in \R\}$  and  $\{ s(x_*+w(s)), s\in(-\e, \e)\}$.
\end{Theorem}

%\textbf{Conflict of Interest:} The authors declare that they have no conflicdt of interest.

\end{document}